\documentclass[a4paper,12pt]{amsart} 
\usepackage[top=25truemm,bottom=20truemm,left=25truemm,right=25truemm]{geometry}
\usepackage{amssymb} 
\usepackage{newtxtext}
\usepackage{xpatch}
\usepackage{txfonts}
\usepackage[all]{xy}
\usepackage{amscd}
\usepackage{amsmath}
\usepackage{mathrsfs}
\usepackage{mathtools}
\usepackage[dvipdfmx]{graphicx}
\mathtoolsset{showonlyrefs=true}

\theoremstyle{normal} 
\newtheorem{Thm}{\bf Theorem}[section]
\newtheorem{Lemma}[Thm]{\bf Lemma}
\newtheorem{Coro}[Thm]{\bf Corollary}
\newtheorem{Prop}[Thm]{\bf Proposition}
\newtheorem{claim}[Thm]{\bf Claim}
\theoremstyle{definition} 
\newtheorem{Def}[Thm]{\bf Definition}
\newtheorem{Rem}[Thm]{\bf Remark}

\newtheorem{Ex}[Thm]{\bf Example}
\newtheorem{Nota}[Thm]{\bf Notation}

\newtheorem{Const}[Thm]{\bf Construction}
\newtheorem*{Const0}{\bf Construction}
\newcommand{\Spec}{\mathrm{Spec}\ }

\newcommand{\Sing}{\mathrm{Sing}\ }

\begin{document}
\title{Jet schemes of singular surfaces of types $D_4^0$ and $D_4^1$ in characteristic $2$}

\author[Y. Koreeda]{Yoshimune Koreeda}

%
\address{
Department of Mathematics, Graduate School of Science \endgraf
Hiroshima University \endgraf
1-3-1 Kagamiyama, Higashi-Hiroshima, 739-8526 \endgraf
Japan
}
\email{y-koreeda@hiroshima-u.ac.jp}


	\maketitle
	\begin{abstract}
		Let $k$ be an algebraically closed field,
		$S$ a variety over $k$ and $m$ a nonnegative integer.
		There is a space $S_m$ over $S$,
		called the jet scheme of $S$ of order $m$,
		parametrizing $m$-th jets on $S$.
		The fiber over the singular locus of $S$ is called the singular fiber.
		
		In this paper,
		we consider the singular fibers of the jet schemes of 2-dimensional rational double points over a field $k$ of characteristic $2$ whose resolution graph is of type $D_4$.
		There are two types of such singularities,
		of type $D_4^0$ and type $D_4^1$.
		We give the irreducible decomposition of the singular fiber and describe the configuration of the irreducible components.
		The case of a $D_4^0$-singularity is quite similar to the case of characteristic $0$ studied in \cite{Ky}.
		The case of $D_4^1$-singularity requires more elaborate analysis of certain subsets of the singular fibers.
	\end{abstract}
\section{Introduction}
Let $k$ be an algebraically closed field of an arbitrary characteristic and $S$ a surface over $k$.
The notion of a jet scheme was introduced by Nash in 1968 in a preprint,
later published in 1995 \cite{Na}.
Let $m$ be a nonnegative integer.
Roughly speaking,
an $m$-th jet of $S$ is an infinitesimal map of order $m$ from a germ of a curve to $S$,
and the scheme parametrizing $m$-th jets is called the $m$-th jet scheme.
We denote the $m$-th jet scheme of $S$ by $S_m$.

For nonnegative integers $m \geq m'$,
there is a natural morphism $\pi_{m,m'}^S : S_m \rightarrow S_{m'}$ called the truncation morphism.
The $0$-th jet scheme is identified with $S$ and we denote the truncation morphism $\pi_{m,0}^S$ by $\pi_m^S$.
We are interested in the fiber $S_m^0$ of $\pi_m^S$ over the singular locus $\Sing S$ of $S$,
which we call the singular fiber of $S$.
In characteristic $0$,
the relation between the singular fibers of jet schemes of surfaces and the exceptional curves of the minimal resolutions of singularities of surfaces was studied in a series of papers by Mourtada \cite{M1,M2} and Mourtada-Pl\'{e}nat \cite{MP}.
For a general surface $S$,
the relation between the irreducible components of $S_m^0$ and the exceptional curves of the minimal resolution of $S$ is not simple.
However,
for rational double points,
Mourtada gave a one-to-one correspondence between the irreducible components of the singular fiber and the exceptional curves of the minimal resolution.

In \cite{Ky},
this correspondence was studied in more detail.
For a fixed order $m$ of the jet scheme of a singular surface of type $A_n$ or type $D_4$,
the configuration of the irreducible components of the singular fiber was investigated.
Furthermore,
a graph was constructed using this information,
and the graph turned out to be isomorphic to the resolution graph of the singularity for a sufficiently large $m$.
Concretely,
the graph is constructed as follows:
Let $Z_m^1,...,Z_m^n$ be the irreducible components of the singular fiber $S_m^0$.
	\begin{Const0}[{\cite[Construction 2.15]{Ky}}]
		Let $V = \{ Z_m^1,...,Z_m^n \}$,
		and let $E \subseteq \{ Z_m^i \cap Z_m^j \mid i,j \in \{ 1,...,n \} ,i \neq j \}$ be the set of the maximal elements for the inclusion relation.
		Then we construct a graph $\Gamma$ as the pair $(V,E)$,
		that is,
		the vertices of $\Gamma$ are elements of $V$ and there is given an edge between $Z_m^i$ and $Z_m^j$ if and only if $Z_m^i \cap Z_m^j \in E$.
	\end{Const0}

In this paper,
we consider $2$-dimensional rational double points in positive characteristics,
especially,
$D_4$-type singular surfaces in characteristic $2$.

First,
we consider the irreducible decomposition of the singular fiber in any characteristic.
For singular surfaces of type $A_n$,
the defining equation of the surface is given by $xy - z^{n+1}$ in any characteristics.
Here and henceforth,
the singular point is at the origin.
In this case the irreducible decomposition of the singular fiber was given by Mourtada in \cite{M1}.
For a singular surface of type $D_n$ for $n \geq 4$,
the defining equation is $x^2 - y^2z + z^n$ in characteristic $0$.
In the first possible case $n=4$,
if the characteristic is a prime different from $2$,
the defining equation is the same as in characteristic $0$.
One can check that the arguments in \cite{M2} work also in this case.

In characteristic $2$,
however,
there are two different types of singularities with the resolution graph of type $D_4$.
One is given by $x^2 + y^2z + yz^2$,
called a singularity of type $D_4^0$,
and the other given by $x^2 + y^2z + yz^2 + xyz$,
called a singularity of type $D_4^1$.
We note that the equation of a singularity of type $D_4^0$ is weighted homogenous while that of a singularity of type $D_4^1$ is not.
In Mourtada \cite{M2},
it was important in giving the irreducible decomposition of the singular fibers that the defining equation is weighted homogenous.
Thus,
for a singular surface of type $D_4^0$,
we will give the irreducible decomposition of the singular fiber using arguments as in Mourtada \cite{M2}.
On the other hand,
for a singular surface of type $D_4^1$,
we have to find another way to prove that certain sets are irreducible.
In this paper,
we do this by a careful study of the codimensions of certain subsets of the singular fiber.

Second,
as for the configuration of the irreducible components of the singular fiber,
we consider the inclusion relations among their intersections as in \cite{Ky}.
For the singular surfaces of type $A_n$,
we gave the irreducible decomposition of the intersections of the irreducible components of the singular fiber in characteristic $0$ in \cite{Ky}.
These decompositions are independent of the characteristics,
so the inclusion relations are the same as in \cite{Ky} in any characteristic.
Next,
we consider the singular surfaces of type $D_4$.
In this case,
while Mourtada gave the irreducible decomposition of the singular fiber in characteristic $0$ in \cite{M2},
generators of defining ideals of irreducible components was not known.
Notwithstanding,
in characteristic $0$,
we could determine the set $E$,
and hence the graph $\Gamma$,
in Construction in \cite{Ky}.
If the characteristic is greater than $2$,
we can determine the configuration using the same arguments as in \cite{Ky} and we do not deal with these cases in this paper.
If the characteristic is $2$,
the arguments need some modification,
and this is the case that we will deal with in this paper.

The following two theorems are the main results in this paper.
First,
we give the irreducible decomposition of the singular fiber.
	\begin{Thm}\label{main theorem1}
		Let $k$ be an algebraically closed field of characteristic $2$ and $S \subset \mathbb{A}^3$ the surface defined by $f = x^2+y^2z+yz^2$ or $g = x^2+y^2z+yz^2+xyz$ in the affine space over $k$.
		If $m \geq 5$,
		the irreducible decomposition of the singular fiber $S_m^0$ is given by
			\begin{center}
				$S_m^0 = Z_m^0 \cup Z_m^1 \cup Z_m^2 \cup Z_m^3$,
			\end{center}
		where $Z_m^0$,
		$Z_m^1$,
		$Z_m^2$ and $Z_m^3$ are as in Definition \ref{既約成分の定義イデアルD_4^0} or Definition \ref{definition of the irreducible components of singular fiber of D_4^1}.
	\end{Thm}
Second,
we describe the inclusion relations between the intersections of irreducible components of the singular fiber.
	\begin{Thm}\label{main theorem2}
		Let $k$ be an algebraically closed field of characteristic $2$,
		$S \subset \mathbb{A}^3$ the surface defined by $f = x^2+y^2z+yz^2$ or $g = x^2+y^2z+yz^2+xyz$ in the affine space over $k$,
		$Z_m^0,...,Z_m^3$ the irreducible components of the singular fiber $S_m^0$ as in Definition \ref{既約成分の定義イデアルD_4^0} or Definition \ref{definition of the irreducible components of singular fiber of D_4^1}.
		\begin{itemize}
			\item[(a)]
				For $0 \leq i < j \leq 3$,
				$Z_m^i \cap Z_m^j \subsetneq Z_m^0$.
			\item[(b)]
				For $1 \leq i, j \leq 3$ with $i \neq j$,
				$Z_m^0 \cap Z_m^i \not\subseteq Z_m^0 \cap Z_m^j$.
			\item[(c)]
				For $1 \leq i, j \leq 3$ with $i \neq j$,
				$Z_m^i \cap Z_m^j \subsetneq Z_m^0 \cap Z_m^i$.
			\item[(d)]
				For $1 \leq i < j \leq 3$ and $1 \leq l \leq 3$,
				$Z_m^0 \cap Z_m^l \not\subseteq Z_m^i \cap Z_m^j$.
		\end{itemize}
	In particular,
	for $0 \leq i<j \leq 3$,
	$Z_m^i \cap Z_m^j$ is maximal in  $\{Z_m^i \cap Z_m^j \mid i, j \in \{0,1,2,3\}, i \neq j\}$ with respect to the inclusion relation if and only if $(i,j) = (0,1), (0,2), (0,3)$.
	\end{Thm}
The next step would be the case of type $D_n$ with $n \geq 5$,
but the calculation becomes more and more difficult.
For instance,
for $n = 5$,
the author could calculate the irreducible components of $S_m^0$ only for $m \leq 8$ in characteristic $0$ using Macaulay2 on a personal computer.
In characteristic $2$,
the calculations become somewhat simpler,
and we hope that the case of characteristic $2$ will give some insight into the characteristic $0$ case.

The organization of this paper is as follows.
In section $2$,
we fix some notations on jet schemes and rational double points in characteristic $2$.
In section $3$ and section $4$,
we give a description of the defining ideals of the irreducible components of the singular fibers of the jet schemes of type $D_4^0$- and $D_4^1$-singularities,
and we study the intersections of the irreducible components to determine the inclusion relations between them.
Then we see that the graphs constructed as in Construction \ref{graph constructed by the singular fiber} are isomorphic to the resolution graph of the singularity for a sufficiently large $m$.

\textrm{\bf{Acknowledgement.}}
The author would like to thank Nobuyoshi Takahashi for valuable advice.
The author would also like to thank the referees of the previous version for careful reading and for valuable advice.
%
%
%
%
%
%
%
%
%
%
%
%
\section{Preliminaries}
In this section,
we recall the definition of the jet schemes and the defining equations of rational double points in $\mathbb{A}^3$ whose resolution graphs are of type $D_4$,
and fix some notations.
\medskip

First of all,
we recall the definition of the jet schemes.
For more details on jet schemes,
refer to \cite{Is2}.
Let $k$ be an algebraically closed field of an arbitrary characteristic,
$X$ a scheme of finite type over $k$ and $m$ a nonnegative integer.
We consider the following functor $F_m^X$.
Let $\mathfrak{Sch}/k$ be the category of schemes over $k$ and $\mathfrak{Set}$ the category of sets.
The functor $F_m^X$ is given by
	\begin{center}
	$F_m^X : \mathfrak{Sch}/k \rightarrow \mathfrak{Set} ; Z \mapsto \mathrm{Hom}_k (Z \times_{\Spec k} \Spec k[t]/\langle t^{m+1} \rangle, X)$.
	\end{center}
The functor $F_m^X$ is representable,
and the object representing $F_m^X$ will be denoted by $X_m$ (\cite[Theorem 2.2]{Is2}, \cite[Proposition 2.2]{Is1}).
	\begin{Def}
		The scheme $X_m$ is called the $m$-th \emph{jet scheme} of $X$.
	\end{Def}
We are interested in a neighborhood of two dimensional isolated hypersurface singularity,
so we consider an affine scheme embedded in $\mathbb{A}^3$ as the target space.
In this case,
we have the following explicit description.
Let us consider a scheme $X$ which is embedded as a hypersurface in an affine space $\mathbb{A}^3$.
Then the affine coordinate ring $\Gamma(X,\mathcal{O}_X)$ of $X$ can be written in the form $k[x,y,z]/\langle f \rangle$.
We introduce some notations.
	\begin{Nota}\label{calculation of jet schemes}
		Let $R_m := k[x_0,...,x_m,y_0,...,y_m,z_0,...,z_m]$,
		$\mathbf{x} := x_0 + x_1t + \cdots + x_mt^m$,
		$\mathbf{y} := y_0 + y_1t + \cdots + y_mt^m$,
		$\mathbf{z} := z_0 + z_1t + \cdots + z_mt^m$
		$\in R_m[t]/\langle t^{m+1} \rangle$.
		For a polynomial $f \in k[x,y,z]$,
		we expand $f(\mathbf{x},\mathbf{y},\mathbf{z})$ as
			\begin{center}
				$f(\mathbf{x},\mathbf{y},\mathbf{z}) = f^{(0)} + f^{(1)}t + \cdots + f^{(m)}t^m$
			\end{center}
		in $R_m[t]/\langle t^{m+1} \rangle$,
		where $f^{(j)} \in R_m$.
		Then the $m$-th jet scheme $X_m$ can be written as
			\begin{center}
				$X_m = \Spec R_m/\langle f^{(0)},...,f^{(m)} \rangle$.
			\end{center}
		For a closed point $\gamma = (a_0,...,a_m,b_0,...,b_m,c_0,...,c_m) \in \mathbb{A}^{3(m+1)} = \Spec R_m$,
		we also write $\gamma = (\sum_{i=0}^m a_it^i,\sum_{i=0}^m b_it^i,\sum_{i=0}^m c_it^i)$.
	\end{Nota}
	\begin{Rem}\label{X cong X_0}
		In this notation,
		$X_0 = \Spec R_0/\langle f^{(0)} \rangle = \Spec k[x_0,y_0,z_0]/\langle f^{(0)} \rangle$.
		Thus,
		$X_0 \cong X$ and so we identify $X_0$ with $X$ in the following.
	\end{Rem}
We note the following fact.
	\begin{Rem}\label{weight of variables}
		If we give the weight $i$ to the variables $x_i, y_i, z_i$ for $0 \leq i \leq m$,
		then the polynomial $f^{(n)}$ is homogenous of degree $n$ for $0 \leq n \leq m$.
		Indeed,
		each term of $\mathbf{x}, \mathbf{y}$ or $\mathbf{z}$ has the same degree in $x_i$, $y_i$, $z_i$ and in $t$,
		and so the same holds for $f(\mathbf{x}, \mathbf{y}, \mathbf{z})$.
		Moreover,
		$f^{(n)}$ is the coefficient of $t^n$ in $f(\mathbf{x}, \mathbf{y}, \mathbf{z})$,
		hence the claim.
	\end{Rem}
In the following sections,
we consider reductions of the polynomials $f^{(i)}$ modulo the ideals generated by $x_j$'s,
$y_j$'s and $z_j$'s for different ranges of $j$.
Hence we set up the following notation.
	\begin{Nota}\label{modulo polynomials}
		Let $m,p,q,r \in \mathbb{Z}_{>0}$.
		For $p,q,r \leq m+1$,
		we set
			\begin{center}
				$L_{pqr} = \langle x_0,...,x_{p-1},y_0,...,y_{q-1},z_0,...,z_{r-1} \rangle \subset R_m$.
			\end{center}
		Moreover,
		let $\mathbf{x}_p = x_pt^p + x_{p+1}t^{p+1} + \cdots + x_mt^m$,
		$\mathbf{y}_q = y_qt^q + y_{q+1}t^{q+1} + \cdots + y_mt^m$ and $\mathbf{z}_r = z_rt^r + z_{r+1}t^{r+1} + \cdots + z_mt^m$.
		(For $p > m$,
		$q > m$ or $r > m$,
		we think of the right hand sides as $0$.)
		For a polynomial $f \in k[x,y,z]$,
		we expand $f(\mathbf{x}_p,\mathbf{y}_q,\mathbf{z}_r)$ as
			\begin{center}
				$f(\mathbf{x}_p,\mathbf{y}_q,\mathbf{z}_r) = f^{(0)}_{pqr} + f^{(1)}_{pqr}t + \cdots + f^{(m)}_{pqr}t^m$
			\end{center}
		in $R_m[t]/\langle t^{m+1} \rangle$,
		where $f^{(j)}_{pqr} \in R_m$.
		Clearly,
		$f^{(j)}_{pqr} \in k[x_p,...,x_m,y_q,...,y_m,z_r,...,z_m]$ holds.
	\end{Nota}
	\begin{Rem}\label{座標関数を含むイデアルの簡略化}
		For $0 \leq j \leq m$,
		we have
			\begin{center}
				$f^{(j)} \equiv f^{(j)}_{pqr} \mod L_{pqr}$.
			\end{center}
		In particular,
			\begin{center}
				$L_{pqr}+\langle f^{(0)},...,f^{(m)} \rangle = L_{pqr} + \langle f_{pqr}^{(0)},...,f_{pqr}^{(m)} \rangle$.
			\end{center}
	\end{Rem}

Next,
we define the truncation morphisms.
Let $m, m' \in \mathbb{Z}_{\geq 0}$ with $m \geq m'$.
	\begin{Def}\label{definition of truncation morpshism}
		The \emph{truncation morphism}
			\begin{center}
				$\pi_{m,m'}^X : X_m \rightarrow X_{m'}$
			\end{center}
		is defined as the morphism induced by the natural morphism $\Spec k[t]/\langle t^{m'+1} \rangle \rightarrow \Spec k[t]/\langle t^{m+1} \rangle$.
		We write
			\begin{center}
				$\varpi_{m,m'}^X : \Gamma(X_{m'},\mathcal{O}_{X_{m'}}) \rightarrow \Gamma(X_m,\mathcal{O}_{X_m})$
			\end{center}
		for the corresponding ring homomorphism.
		When $X = \mathbb{A}^3$,
		we use $\pi_{m,m'}$ (resp. $\varpi_{m,m'}$) instead of $\pi_{m,m'}^{\mathbb{A}^3}$ (resp. $\varpi_{m,m'}^{\mathbb{A}^3}$).
		We write $\pi_m^X$ (resp. $\varpi_{m}^X$) for $\pi_{m,0}^X$ (resp. $\varpi_{m,0}^X$) and regard it as a morphism from $X_m$ to $X$ (resp. $\Gamma(X,\mathcal{O}_{X})$ to $\Gamma(X_m,\mathcal{O}_{X_m})$).
	\end{Def}
We are interested in the fiber of the truncation morphism at a singular point,
so we introduce the following term.
	\begin{Def}
		Let $X \subseteq \mathbb{A}^3$ be a surface with an isolated singular point at the origin $0$ and $m$ a positive integer.
		The fiber $\pi_{m}^{-1}(0)$ of the truncation morphism at the singular point is called the \emph{singular fiber} and is denoted by $X_m^0$.
	\end{Def}

The following remark explains the relation between ideals in $\Gamma(X_m,\mathcal{O}_{X_m})$ and $\Gamma((\mathbb{A}^3)_m,\mathcal{O}_{(\mathbb{A}^3)_m})$.
	\begin{Rem}\label{relations between Gamma(X_m) and affine coordinate}
		Recall that
			\begin{center}
				$R_m := \Gamma((\mathbb{A}^3)_m,\mathcal{O}_{(\mathbb{A}^3)_m}) = k[x_0,...,x_m,y_0,...,y_m,z_0,...,z_m]$.
			\end{center}
		Let $X$ be $\mathbf{V}(f)$,
		$i_m : X_m \rightarrow (\mathbb{A}^3)_m$ the natural inclusion and $\iota_m : R_m \rightarrow \Gamma(X_m,\mathcal{O}_{X_m})$ the corresponding ring homomorphism.
		For any $I \subset R_m$ with $I \supset \langle f^{(0)},...,f^{(m)} \rangle$,
		we set $\tilde{I} := \iota_m(I)$.
		Then,
		$i_m(\mathbf{V}(\tilde{I})) = \mathbf{V}(I)$ clearly holds.
		Under this inclusion morphism $i_m$,
		$X_m$ and its closed subschemes are identified with closed subschemes in $(\mathbb{A}^3)_m$.
		
		We can describe the inverse images by the truncation morphism are as follows:
		Suppose $m \geq m'$ and let $Z \subseteq X_{m'}$ be a closed subscheme defined by $\tilde{I} \subset \Gamma(X_{m'},\mathcal{O}_{X_{m'}})$,
		$I \subset R_{m'}$ an ideal with $I \supset \langle f^{(0)},...,f^{(m')} \rangle$ and $\iota_{m'}(I) = \tilde{I}$.
		Then a defining ideal of $(\pi_{m,m'}^X)^{-1}(Z) = (\pi_{m,m'}^X)^{-1}(\mathbf{V}(I))$ is $\varpi_{m,m'}^X(\tilde{I}) \cdot \Gamma(X_m, \mathcal{O}_{X_m})$ and this ideal satisfies
			\begin{alignat*}{5}
				\varpi_{m,m'}^X(\tilde{I})\cdot \Gamma(X_{m},\mathcal{O}_{X_m}) 	&= \iota_m(\varpi_{m,m'}(I)\cdot R_{m} +\langle f^{(m'+1)},f^{(m'+2)},...,f^{(m)} \rangle).
			\end{alignat*}
		Hence,
		under the above identification,
		the defining ideal of $(\pi_{m,m'}^X)^{-1}(Z) \subseteq (\mathbb{A}^3)_m$ is $\varpi_{m,m'}(I)\cdot R_m +\langle f^{(m'+1)},f^{(m'+2)},...,f^{(m)} \rangle$ in $R_m$.
	\end{Rem}
Throughout this paper,
we identify $X_m$ and its closed subschemes with the corresponding subschemes in $(\mathbb{A}^3)_m$ by $i_m$.

The following proposition and remark describe the inverse images of the smooth locus by the truncation morphism.
	\begin{Prop}[{\cite[Proposition 2.4]{Is2}}]\label{非特異部分の切り詰め射による逆像}
		Let $X, Y$ be schemes over $k$.
		If $f : X \rightarrow Y$ is an \'{e}tale morphism,
		then $X_m \cong Y_m \times_Y X$ for every $m \in \mathbb{Z}_{\geq 0}$.
	\end{Prop}
	\begin{Rem}\label{非特異部分の逆像}
		If $X$ is a $n$-dimensional variety,
		then $\pi_m : (X_{\mathrm{sm}})_m \rightarrow X_{\mathrm{sm}}$ is a locally trivial fibration with the fiber $\mathbb{A}^{nm}$
		by Proposition \ref{非特異部分の切り詰め射による逆像},
		where $X_{\mathrm{sm}} = X - \Sing X$.
		Hence $\overline{\pi_m^{-1}(X_{\mathrm{sm}})}$ is an irreducible component of $X_m$,
		and we call it the main component.
	\end{Rem}
In general,
the $m$-th jet scheme $X_m$ is not irreducible.
However,
if $X$ has only rational double points,
then $X_m$ is irreducible and consists of the main component (\cite[Corollary 10.2.9]{Is}).

Next,
we recall the defining equations of singularities in $\mathbb{A}^3$ whose resolution graphs are of type $D_4$.
In positive characteristic other than $2$,
a surface singularity of type $D_4$ can be defined by $x^2 - y^2z+z^3$,
which is the same as in the case of characteristic $0$.
On the other hand,
in characteristic $2$,
there are two singularities (\cite[Section 3]{Ar}):
a singularity of type $D_4^0$,
defined by
	\begin{alignat*}{5}
		f &= x^2 + y^2z + yz^2,
	\end{alignat*}
and a singularity of type $D_4^1$,
defined by
	\begin{align*}
		g &= x^2 + y^2z + yz^2 + xyz.
	\end{align*}
In characteristic different from $2$,
they both give a singularity of type $D_4$.

To conclude this section,
we give the lemmas that will be used in the following sections.

One key point in the description of the singular fibers in \cite{M2} was that $f^{(j)}$ is often of the form $Ay_i + B$ (resp. $Az_i + B$) where $A$ and $B$ do not contain $y_l$ (resp. $z_l$) with $l \geq i$.
Hence we set up the following notation.
	\begin{Nota}\label{Leading term for variables}
		Let $h \in R_m$.
		The sum of the terms in $h$ containing $y_i$ (resp. $z_i$) with the largest index $i$ is denoted by $T_y(h)$ (resp. $T_z(h)$).
	\end{Nota}
	\begin{Ex}
		If $h = x_3^2 + y_2^2z_2 + y_2z_2^2$,
		then $T_y(h) = y_2^2z_2 + y_2z_2^2$.
		If $h = x_3^2 + y_2^2z_2 + y_2z_2^2 + y_4z_1^2$,
		then $T_y(h) = y_4z_1^2$.
	\end{Ex}
For any polynomials $f$ and $g$,
we have the following.
	\begin{Lemma}	\label{reduction of f^{(l)} g^{(l)} modulo L_{pqr}}
		Let $k$ be a field of characteristic $2$ and $f, g \in k[x,y,z]$ be defined by $f = x^2+y^2z+yz^2$ and $g = x^2+y^2z+yz^2+xyz$.
		Assume $p,q,r \in \mathbb{Z}_{>0}$ and $l \in \mathbb{Z}_{\geq 0}$ with $l \leq m$.
		Then we have
			\begin{center}
				$\displaystyle f_{pqr}^{(l)} = \sum_{u\geq p, 2u=l} x_u^2 + \sum_{v\geq q, w\geq r, 2v+w=l} y_v^2z_w + \sum_{v\geq q, w\geq r, v+2w=l} y_vz_w^2$
			\end{center}
and
			\begin{center}
				$\displaystyle g_{pqr}^{(l)} = \sum_{u\geq p, 2u=l} x_u^2 + \sum_{v\geq q, w\geq r, 2v+w=l} y_v^2z_w + \sum_{v\geq q, w\geq r, v+2w=l} y_vz_w^2 + \sum_{u\geq p, v\geq q, w\geq r, u+v+w=l} x_uy_vz_w$.
			\end{center}
		Here,
		if there are no $u,v$ and $w$ satisfying the conditions,
		we regard the sums as $0$.
		Furthermore,
		the following hold.
			\begin{enumerate}
				\item[(a)]
					If $l < 2p$,
					$l < 2q + r$ and $l < q+2r$,
					then we have
						\begin{center}
							$f_{pqr}^{(l)} = g_{pqr}^{(l)} = 0$.
						\end{center}
				\item[(b)]
					If $l = 2p$,
					$l < 2q + r$ and $l < q+2r$,
					then we have
						\begin{center}
							$f_{pqr}^{(l)} = g_{pqr}^{(l)} = x_p^2$.
						\end{center}
				\item[(c)]
					If $(p >) q=r$ and $l = 2p = 3q$,
					then
						\begin{center}
							$f_{pqq}^{(l)} = g_{pqq}^{(l)} = x_p^2+y_q^2z_q+y_qz_q^2$.
						\end{center}
				\item[(d)]
					If $(p >) q = r$ and $l = 3q < 2p$,
					then
						\begin{center}
							$f_{pqq}^{(l)} = g_{pqq}^{(l)} = y_q^2z_q + y_qz_q^2 = y_qz_q(y_q+z_q)$.
						\end{center}
				\item[(e)]
					If $p \geq q > r$ and $l \geq 2p = q+2r$,
					then
						\begin{center}
							$T_y(f_{pqr}^{(l)}) = T_y(g_{pqr}^{(l)}) = y_{l-2r}z_r^2$.
						\end{center}
				\item[(f)]
					If $p \geq r > q$ and $l \geq 2p = 2q+r$,
					then
						\begin{center}
							$T_z(f_{pqr}^{(l)}) = T_z(g_{pqr}^{(l)}) = y_q^2z_{l-2q}$.
						\end{center}
				\item[(g)]
					If $p > q=r$ and $l > 3q$,
					then
						\begin{center}
							$T_y(f_{pqq}^{(l)}) = T_y(g_{pqq}^{(l)}) = y_{l-2q}z_q^2$.
						\end{center}
				\item[(h)]
					If $p > q=r$ and $l > 3q$,
					then
						\begin{center}
							$T_z(f_{pqq}^{(l)}) = T_z(g_{pqq}^{(l)}) = y_q^2z_{l-2q}$.
						\end{center}
			\end{enumerate}
	\end{Lemma}
	\begin{Rem}\label{remark of the difference between f and g}
		We note that if $p,q,r$ satisfy the above conditions (a)--(d),
		then there are no terms coming from $xyz$ in $g^{(l)}_{pqr}$.
		In particular,
		the polynomial $g_{pqq}^{(2p)} = x_p^2+y_q^2z_q+y_qz_q^2$ appearing in (c) has the same form as the defining equation $f$ of a singular surface of type $D_4^0$.
		Since the $D_4^0$-type singular surface is irreducible,
		$g_{pqq}^{(2p)}$ is also irreducible in $k[x_p,y_q,z_q]$.
		Moreover,
		if $p,q,r$ satisfy the above conditions (e)--(h),
		then there are no terms coming from $xyz$ in $T_y(g^{(l)}_{pqr})$ and $T_z(g_{pqr}^{(l)})$.
	\end{Rem}
	\begin{proof}
		First of all,
		we note that
			\begin{align*}
				\mathbf{x}_p^2	&= (x_pt^p + x_{p+1}t^{p+1} + \cdots + x_mt^m)^2\\
								&= x_p^2t^{2p} + x_{p+1}^2t^{2(p+1)} + \cdots + x_m^2t^{2m},
			\end{align*}
		and similarly for $\mathbf{y}_q^2$ and $\mathbf{z}_r^2$,
		since we are working in characteristic $2$.
		
		Since $f_{pqr}^{(l)}$ and $g_{pqr}^{(l)}$ are the coefficient of $t^l$ in the expansion of $f(\mathbf{x}_p, \mathbf{y}_q, \mathbf{z}_r) = \mathbf{x}_p^2 + \mathbf{y}_q^2\mathbf{z}_r + \mathbf{y}_q\mathbf{z}_r^2$ and $g(\mathbf{x}_p, \mathbf{y}_q, \mathbf{z}_r) = \mathbf{x}_p^2 + \mathbf{y}_q^2\mathbf{z}_r + \mathbf{y}_q\mathbf{z}_r^2 + \mathbf{x}_p\mathbf{y}_q\mathbf{z}_r$,
		we obtain
			\begin{center}
				$\displaystyle f_{pqr}^{(l)} = \sum_{u\geq p, 2u=l} x_u^2 + \sum_{v\geq q, w\geq r, 2v+w=l} y_v^2z_w + \sum_{v\geq q, w\geq r, v+2w=l} y_vz_w^2$
			\end{center}
		and
			\begin{center}
				$\displaystyle g_{pqr}^{(l)} = \sum_{u\geq p, 2u=l} x_u^2 + \sum_{v\geq q, w\geq r, 2v+w=l} y_v^2z_w + \sum_{v\geq q, w\geq r, v+2w=l} y_vz_w^2 + \sum_{u\geq p, v\geq q, w\geq r, u+v+w=l} x_uy_vz_w$
			\end{center}
		by a direct calculation.
		
		Before we begin the proof of (a)--(h),
		we note that the difference between $f^{(l)}$ and $g^{(l)}$ is the terms coming from $xyz$.
		As we saw in Remark \ref{remark of the difference between f and g},
		the results of Lemma \ref{reduction of f^{(l)} g^{(l)} modulo L_{pqr}} are the same for $f$ and $g$.
		Noting that every term in $(xyz)^{(l)}$ is of the form $x_u^iy_v^jz_w^k$ with $i,j,k > 0$ and that $f^{(l)}$ contains no such terms,
		it suffices to prove Lemma \ref{reduction of f^{(l)} g^{(l)} modulo L_{pqr}} for $g$.
				
		Let us prove (a),
		(b),
		(c) and (d).
		Note,
		first of all,
		that the lowest exponent of $t$ appearing in $\mathbf{x}_p^2$ (resp. $\mathbf{y}_q^2\mathbf{z}_r$,
		$\mathbf{y}_q\mathbf{z}_r^2$,
		$\mathbf{x}_p\mathbf{y}_q\mathbf{z}_r$) is $2p$ (resp. $2q+r$, $q+2r$, $p+q+r$).
		
		For (a),
		we only have to show that the lowest exponent of $t$ in $g(\mathbf{x}_p, \mathbf{y}_q, \mathbf{z}_r)$ is greater than $l$.
		By the conditions in (a),
		we have $l < 2p$,
		$l < 2q+r$ and $l < q+2r$,
		so we only have to check $l < p+q+r$.
		From the assumption,
		we have $p > l/2$ and $q+r > 2l/3$,
		and then that $p+q+r > 7l/6 > l$.
		Thus $g_{pqr}^{(l)} = 0$ holds.
		
		For (b),
		again we only have to show that $p+q+r > l$.
		By assumption in (b) $p = l/2$ and $q+r > 2l/3$,
		so $p+q+r > 7l/6 > l$.
		
		For (c),
		we have $l = 2p = 3q$ by assumption and $\mathbf{x}_p^2$,
		$\mathbf{y}_q^2\mathbf{z}_q$ and $\mathbf{y}_q\mathbf{z}_q^2$ give rise to terms $x_p^2$,
		$y_q^2z_q$ and $y_qz_q^2$,
		so we only have to show that $p+q+q > l$.
		From the assumption $p > q$,
		so $l = 3q < p+q+q$.
		Hence $g_{pqq}^{(l)} = x_p^2 + y_q^2z_q + y_qz_q^2$.
		
		For (d),
		we have $2p > l$ and $p>q$ by assumption,
		so $p+q+q > 3q = l$.
		Therefore,
		$g_{pqq}^{(3q)} = y_q^2z_q+y_qz_q^2 = y_qz_q(y_q+z_q)$ holds.
		
		Before proving the statements (e)--(h),
		let us find out the terms containing $y_i$ (resp. $z_i$) with the largest $i$ in $(\mathbf{y}_q^2\mathbf{z}_r)^{(l)}$,
		$(\mathbf{y}_q\mathbf{z}_r^2)^{(l)}$ and $(\mathbf{x}_p\mathbf{y}_q\mathbf{z}_r)^{(l)}$,
		where we write the coefficient of $t^l$ in $\mathbf{y}_q^2\mathbf{z}_r$ as $(\mathbf{y}_q^2\mathbf{z}_r)^{(l)}$,
		and so on.
		
		First,
		we look at $T_y((\mathbf{y}_q^2\mathbf{z}_r)^{(l)})$ (resp. $T_z((\mathbf{y}_q\mathbf{z}_r^2)^{(l)}))$ (see Notation \ref{Leading term for variables}).
		If $l < 2q+r$ (resp. $l < q+2r$),
		it is obvious that $(\mathbf{y}_q^2\mathbf{z}_r)^{(l)} = 0$ (resp. $(\mathbf{y}_q\mathbf{z}_r^2)^{(l)} = 0$),
		so we assume $l \geq 2q+r$ (resp. $l \geq q+2r$).
		If $l-r$ (resp. $l-q$) is even,
		then
			\begin{center}
				$T_y((\mathbf{y}_q^2\mathbf{z}_r)^{(l)}) = y_{\frac{l-r}{2}}^2z_r$ \ (resp. $T_z((\mathbf{y}_q\mathbf{z}_r^2)^{(l)}) = y_qz_{\frac{l-q}{2}}^2$),
			\end{center}
		and if $l-r$ (resp. $l-q$) is odd,
		then
			\begin{center}
				$T_y((\mathbf{y}_q^2\mathbf{z}_r)^{(l)}) = y_{\frac{l-r-1}{2}}^2z_{r+1}$ \ (resp. $T_z((\mathbf{y}_q\mathbf{z}_r^2)^{(l)}) = y_{q+1}z_{\frac{l-q-1}{2}}^2$).
				\end{center}
		Second,
		assuming $l \geq q+2r$ (resp. $l \geq 2q+r$),
			\begin{center}
				$T_y((\mathbf{y}_q\mathbf{z}_r^2)^{(l)}) = y_{l-2r}z_r^2$ \ (resp. $T_z((\mathbf{y}_q^2\mathbf{z}_r)^{(l)}) = y_q^2z_{l-2q}$).
			\end{center}
		Third,
		if $l \geq p+q+r$,
		then
			\begin{center}
				$T_y((\mathbf{x}_p\mathbf{y}_q\mathbf{z}_r)^{(l)}) = x_py_{l-p-r}z_r$ \ (resp. $T_z((\mathbf{x}_p\mathbf{y}_q\mathbf{z}_r)^{(l)}) = x_py_qz_{l-p-q}$),
			\end{center}
		and if $l < p+q+r$,
		then $(\mathbf{x}_p\mathbf{y}_q\mathbf{z}_r)^{(l)} = 0$.
			
		Now,
		we prove (e) (resp. (g)).
		Since the terms $T_y((\mathbf{y}_q^2\mathbf{z}_r)^{(l)})$,
		$T_y((\mathbf{y}_q\mathbf{z}_r^2)^{(l)})$ and $T_y((\mathbf{x}_p\mathbf{y}_q\mathbf{z}_r)^{(l)})$ are as above,
		we only have to show that $l-2r > (l-r)/2$ and $l-2r > l-p-r$.
		First,
		we prove $l-2r > (l-r)/2$.
		By the assumption $q > r$ and $l \geq q+2r$ (resp. $l > 3q = 3r$),
		so $l > 3r$ and this is equivalent to $l-2r > (l-r)/2$.
		Moreover,
		$p > r$ by the assumption.
		So $l-2r > l-p-r$.
		Thus,
		$T_y(g_{pqr}^{(l)}) = y_{l-2r}z_{r}^2$ (resp. $T_y(g_{pqq}^{(l)}) = y_{l-2q}z_{q}^2 = y_{l-2r}z_{r}^2$).
		
		By symmetry,
		we also have (f) and (h).
	\end{proof}
Focusing on $g$,
we have the following lemma.
	\begin{Lemma}\label{modulo prop focus on g}
		Assume $p,q,r \in \mathbb{Z}_{>0}$ and $l \in \mathbb{Z}_{\geq 0}$ with $l \leq m$.
			\begin{itemize}
				\item[(a,g)]
					If $l \notin 2\mathbb{Z}$ and $l < 2q+r$, $l < q+2r$ and $l < p+q+r$,
					then we have
						\begin{center}
							$g_{pqr}^{(l)} = 0$.
						\end{center}
				\item[(b,g)]
					If $l = 2p'$ with $p \leq p'$,
					$l < 2q+r$,
					$l < q+2r$ and $l < p+q+r$,
					then we have
						\begin{center}
							$g_{pqr}^{(l)} = x_{p'}^2$.
						\end{center}
			\end{itemize}
	\end{Lemma}
The proof is basically the same as that of Lemma \ref{reduction of f^{(l)} g^{(l)} modulo L_{pqr}}(a) and (b).

The next lemma is one of the key points in proving the irreducibility of the certain closed subsets of the singular fiber.
	\begin{Lemma}\label{codimension conditions}
		Let $S$ be the surface defined by $f$ or $g$,
		$S_m$ the $m$-th jet scheme of $S$ and $S_m^0$ the singular fiber for $m \geq 1$.
		\begin{itemize}
			\item[(a)]
				If $Z \subseteq S_m$ is an irreducible component,
				then
				$\mathrm{codim}_{(\mathbb{A}^3)_m}\ Z \leq m+1$ (or equivaelntly $\mathrm{dim}\ Z \geq 2m+2$).
			\item[(b)]
				If $Z \subseteq S_m^0$ is an irreducible component,
				then
				$\mathrm{codim}_{(\mathbb{A}^3)_m}\ Z \leq m+2$ (or equivaelntly $\mathrm{dim}\ Z \geq 2m+1$).
		\end{itemize}
	\end{Lemma}
\begin{proof}
(a)
The $m$-th jet scheme $S_m$ is defined by the ideal generated by $m+1$ elements $\{ f^{(0)},...,f^{(m)} \}$,
so for any irreducible component $Z \subseteq S_m$,
	\begin{center}
		$\mathrm{codim}_{(\mathbb{A}^3)_m}\ Z \leq m+1$.
	\end{center}
by Krull's height theorem.\\
(b)
The singular fiber $S_m^0$ is defined by
	\begin{center}
		$L_{111} + \langle f^{(0)},...,f^{(m)} \rangle = L_{111} + \langle f_{111}^{(0)},...,f_{111}^{(m)} \rangle$\\
		(resp. $L_{111} + \langle g^{(0)},...,g^{(m)} \rangle = L_{111} + \langle g_{111}^{(0)},...,g_{111}^{(m)} \rangle$),
	\end{center}
from Remark \ref{relations between Gamma(X_m) and affine coordinate} and Remark \ref{座標関数を含むイデアルの簡略化}.
By Lemma \ref{reduction of f^{(l)} g^{(l)} modulo L_{pqr}}(a),
we have
	\begin{center}
		$f^{(0)}_{111} = f^{(1)}_{111} = 0$ (resp. $g^{(0)}_{111} = g^{(1)}_{111} = 0$)
	\end{center}
and $S_m^0$ is defined by the ideal generated by $3+m+1-2 = m+2$ elements.
Thus as in (a),
for any irreducible component $Z \subseteq S_m^0$,
	\begin{center}
		$\mathrm{codim}_{(\mathbb{A}^3)_m}\ Z \leq m+2$.
	\end{center}
\end{proof}
%
%
%
%
%
%
%
%
%
\section{Jet schemes of a singular surface of type $D_4^0$}
In this section,
we deal with a singular surface of type $D_4^0$.
First,
we find the irreducible decomposition of the singular fiber.
This can be done by the method of Mourtada \cite{M2}.
Next,
we determine the inclusion relations among intersections of irreducible components of the singular fiber.
\medskip

For the determination of the irreducible decomposition in the $D_4^0$ case,
the arguments are almost the same as in characteristic $0$.
	\begin{Rem}
	In positive characteristic not equal to $2$,
	a singular surface of type $D_4$ is defined by $h = x^2 - y^2z+z^3$ in $\mathbb{A}^3$.
	We consider the transformation
		\[
			\begin{cases}
			x \mapsto x, \\
			y \mapsto \frac{1}{\sqrt[6]{4}}y + \frac{\sqrt[3]{4}}{2}z, \\
			z \mapsto -\frac{1}{\sqrt[6]{4}}y + \frac{\sqrt[3]{4}}{2}z.
			\end{cases}
		\]
	Then the polynomial $f = x^2+y^2z+yz^2$ is mapped to $h$.
	\end{Rem}

Let $S = \mathbf{V}(f) \subseteq \mathbb{A}^3$ be a singular surface of type $D_4^0$ over an algebraically closed field $k$ of characteristic $2$,
$S_m$ the $m$-th jet scheme of $S$,
$S_m^0$ the singular fiber of $S_m$ and $R_m = k[x_0,...,x_m,y_0,...,y_m,z_0,...,z_m]$.
Moreover,
we set
	\begin{center}
		$L_{pqr} = \langle x_0,...,x_{p-1},y_0,...,y_{q-1},z_0,...,z_{r-1} \rangle$
	\end{center}
for positive integers $p,q,r \in \mathbb{Z}_{> 0}$ with $p,q,r \leq m+1$.
	\begin{Def}
		For $m \geq 1$,
		we define the following ideals in $R_m$:
			\begin{alignat*}{11}
				J_m^1 &= L_{211} &&+ \langle y_1 \rangle &&+ \langle f^{(0)},...,f^{(m)} \rangle &&= L_{221} &&+ \langle f^{(0)},...,f^{(m)} \rangle,\\
				J_m^2 &= L_{211} &&+ \langle z_1 \rangle &&+ \langle f^{(0)},...,f^{(m)} \rangle &&= L_{212} &&+ \langle f^{(0)},...,f^{(m)} \rangle,\\
				J_m^3 &= L_{211} &&+ \langle y_1+z_1 \rangle &&+ \langle f^{(0)},...,f^{(m)} \rangle.
			\end{alignat*}
	\end{Def}
By Remark \ref{relations between Gamma(X_m) and affine coordinate} and $J_m^i \supset L_{111}$,
these ideals include the defining ideal of the singular fiber,
and hence correspond to closed subsets in the singular fiber $S_m^0$.

We have the following symmetries:
	\begin{Nota}\label{symmetric morphisms in D_4^0}
		Let $\varphi_1$ and $\varphi_2$ be the automorphisms of $R_m$ defined by
		\[
		\varphi_1 :
		\begin{cases}
			x_i \mapsto x_i, \\
			y_i \mapsto z_i, \\
			z_i \mapsto y_i,
		\end{cases}
		\varphi_2 :
		\begin{cases}
			x_i \mapsto x_i, \\
			y_i \mapsto y_i, \\
			z_i \mapsto y_i + z_i.
		\end{cases}
		\]
		These automorphisms $\varphi_1$ and $\varphi_2$ induce ring isomorphisms $(\varphi_1)_{z_1} : (R_m)_{z_1} \rightarrow (R_m)_{y_1}$ and $(\varphi_2)_{y_1} : (R_m)_{y_1} \rightarrow (R_m)_{y_1}$.
		We write the isomorphisms corresponding to $\varphi_1$, $\varphi_2$, $(\varphi_1)_{z_1}$ and $(\varphi_2)_{y_1}$ as $\psi_1$, $\psi_2$, $(\psi_1)_{z_1}$ and $(\psi_2)_{y_1}$,
		respectively.
		For simplicity,
		we write $(\varphi_1)_{z_1}$, $(\varphi_2)_{y_1}$, $(\psi_1)_{z_1}$ and $(\psi_2)_{y_1}$ as $\varphi_1$, $\varphi_2$, $\psi_1$ and $\psi_2$.
	\end{Nota}
	\begin{Lemma}\label{symmetric morphisms}
		\begin{enumerate}
		\item[(a)]
			For $m \geq 1$, $i \in \{ 0,...,m \}$ and $k = 1,2$,
			$\varphi_k$ preserve $f^{(i)}$ i.e.,
				\begin{center}
					$\varphi_k(f^{(i)}) = f^{(i)}$
				\end{center}
			in $R_m$.
			In particular,
			the morphisms $\psi_k$ preserve $S_m$.
		\item[(b)]
			For $m \geq 1$,
				\begin{center}
					$\varphi_1(J_m^1\cdot (R_m)_{z_1}) = J_m^2\cdot (R_m)_{y_1}$
				\end{center}
			and
				\begin{center}
					$\varphi_2(J_m^2\cdot (R_m)_{y_1}) = J_m^3\cdot (R_m)_{y_1}$.
				\end{center}
		\item[(c)]
			The morphisms $\varphi_1$, $\varphi_2$, $\psi_1$ and $\psi_2$ preserve the union,
			the intersection and the inclusion relations of sets.
		\end{enumerate}
	\end{Lemma}	
	\begin{proof}
		(a)
			Note that the automorphisms $\varphi_1$ and $\varphi_2$ are induced by the automorphisms of $k[x,y,z]$ defined by
			\[
				\overline{\varphi_1} :
				\begin{cases}
					x \mapsto x, \\
					y \mapsto z, \\
					z \mapsto y,
				\end{cases}
				\overline{\varphi_2} :
				\begin{cases}
				x \mapsto x, \\
				y \mapsto y, \\
				z \mapsto y + z,
				\end{cases}
			\]
			and
			$\overline{\varphi_1}(f) = x^2+z^2y+zy^2 = f$ and
				\begin{alignat*}{5}
					\overline{\varphi_2}(f)	&= x^2+y^2(y+z)+y(y+z)^2\\
											&= x^2+y(y+z)(y+y+z)\\
											&= x^2+y(y+z)z 			&= x^2+y^2z+yz^2,
				\end{alignat*}
			where $y+y = 0$ since $k$ is a field of characteristic $2$.
			Since
				\begin{alignat*}{7}
					\displaystyle f\left(\sum_{i=0}^m x_it^i, \sum_{i=0}^m y_it^i, \sum_{i=0}^m z_it^i\right)
					&= \overline{\varphi_k}(f)\left(\sum_{i=0}^m x_it^i, \sum_{i=0}^m y_it^i, \sum_{i=0}^m z_it^i\right)\\
					&= f\left(\sum_{i=0}^m \varphi_k(x_i)t^i, \sum_{i=0}^m \varphi_k(y_i)t^i, \sum_{i=0}^m \varphi_k(z_i)t^i\right)\\
					&= \sum_{i=0}^m \varphi_k(f^{(i)})t^i
				\end{alignat*}
			in $R_m[t]/\langle t^{m+1} \rangle$ for $k= 1,2$,
			$\varphi_k$ preserves the polynomials $f^{(i)}$ ($i \in \{0,...,m\}$),
			and the induced automorphism $\psi_k$ preserves $S_m$.
		\medskip\\
		(b)
			We can easily check that
				\begin{center}
					$\varphi_1((L_{211} + \langle y_1 \rangle)\cdot (R_m)_{z_1}) = (L_{211} + \langle z_1 \rangle)\cdot (R_m)_{y_1}$
				\end{center}
			and
				\begin{center}
					$\varphi_2((L_{211} + \langle z_1 \rangle)\cdot (R_m)_{y_1}) = (L_{211} + \langle y_1+z_1 \rangle)\cdot (R_m)_{y_1}$.
				\end{center}
			By the assertion (a),
			we have $\varphi_1(J_m^1\cdot (R_m)_{z_1}) = J_m^2\cdot (R_m)_{y_1}$ and $\varphi_2(J_m^2\cdot (R_m)_{y_1}) = J_m^3\cdot (R_m)_{y_1}$.
		\medskip\\
		(c)
			The morphisms
			$\varphi_1$, $\varphi_2$, $\psi_1$ and $\psi_2$ are isomorphisms,
			so under these morphisms,
			the union,
			the intersection and the inclusion relations of sets are preserved.
	\end{proof}
We will show that $J_m^1, J_m^2$ and $J_m^3$ define irreducible components of certain open subsets of $S_m^0$.
	\begin{Prop}\label{primeness of the open ideals}
	For $m \geq 3$,
	the ideal $J_m^1\cdot (R_m)_{z_1}$ is a prime ideal in $(R_m)_{z_1}$,
	and $J_m^2\cdot (R_m)_{y_1}$ and $J_m^3\cdot (R_m)_{y_1}$ are prime ideals in $(R_m)_{y_1}$.
	Moreover,
	the heights of the ideals $J_m^1\cdot (R_m)_{z_1}$,
	$J_m^2\cdot (R_m)_{y_1}$ and $J_m^3\cdot (R_m)_{y_1}$ are $m+2$.
	\end{Prop}
	\begin{proof}
	We prove that the ideal $J_m^1\cdot (R_m)_{z_1}$ is prime of height $m+2$.
	By Lemma \ref{symmetric morphisms}(b),
	it follows that $J_m^2\cdot (R_m)_{y_1}$ and $J_m^3\cdot (R_m)_{y_1}$ are prime ideals of height $m+2$ in $(R_m)_{y_1}$.
	
	First,
	we note that
		\begin{center}
			$J_m^1 = L_{221} + \langle f^{(0)},...,f^{(m)} \rangle = L_{221} + \langle f_{221}^{(0)},...,f_{221}^{(m)} \rangle$
		\end{center}
	by Remark \ref{座標関数を含むイデアルの簡略化}.
	Then,
	we have
		\begin{alignat*}{7}
			f^{(0)}_{221} &= f^{(1)}_{221} = f^{(2)}_{221} = f^{(3)}_{221} = 0
		\end{alignat*}
	by Lemma \ref{reduction of f^{(l)} g^{(l)} modulo L_{pqr}}(a) and
		\begin{alignat*}{7}
			T_y(f^{(l)}_{221}) = y_{l-2}z_1^2
		\end{alignat*}
	for $4 \leq l \leq m$ by Lemma \ref{reduction of f^{(l)} g^{(l)} modulo L_{pqr}}(e).
	Hence there exists $h_l \in k[x_2,...,x_m,y_2,...,y_{l-3},z_1,...,z_m]$ such that
		\begin{center}
			$f_{221}^{(l)} = y_{l-2}z_1^2 + h_l$
		\end{center}
	for $4 \leq l \leq m$.
	Since
		\begin{center}
			$\displaystyle y_{l-2} + \frac{h_l}{z_1^2} = \frac{1}{z_1^2}f_{221}^{(l)} \in J_m^1\cdot (R_m)_{z_1}$,
		\end{center}
	we have
		\begin{alignat*}{7}
			J_m^1\cdot (R_m)_{z_1}	&= (L_{221}+\langle f^{(0)}_{221},...,f^{(m)}_{211} \rangle)\cdot(R_m)_{z_1}\\
									&= \left(L_{221}+\left\langle y_{2} + \frac{h_4}{z_1^2},...,y_{m-2} + \frac{h_m}{z_1^2} \right\rangle \right)\cdot(R_m)_{z_1}.
		\end{alignat*}
	Thus the ideal $J_m^1\cdot(R_m)_{z_1}$ is a prime ideal of $(R_m)_{z_1}$ and the height of $J_m^1\cdot(R_m)_{z_1}$ is $m+2$.
	\end{proof}

Let us define some ideals of $R_m$ and the corresponding closed subsets of $(\mathbb{A}^3)_m$.
	\begin{Def}\label{既約成分の定義イデアルD_4^0}
		For $m \geq 1$,
		we define
			\begin{alignat}{5}
				I_m^0 &= L_{222} + \langle f^{(0)},...,f^{(m)} \rangle,\\
				I_m^1 &= J_m^1\cdot(R_m)_{z_1} \cap R_m,\\
				I_m^2 &= J_m^2\cdot(R_m)_{y_1} \cap R_m,\\
				I_m^3 &= J_m^3\cdot(R_m)_{y_1} \cap R_m.
			\end{alignat}
		Furthermore,
		we define closed subsets
			\begin{align}
				Z_m^i &:= \mathbf{V}(I_m^i) \label{define the irreducible components for a D_4^0-type}
			\end{align}
		for $0 \leq i \leq 3$.
	\end{Def}

	\begin{Rem}\label{the codimension of the irreducible components}
		\begin{itemize}
			\item[(a)]
				By Proposition \ref{primeness of the open ideals},
				the ideals $I_m^1$,
				$I_m^2$ and $I_m^3$ for $m \geq 3$ are prime of height $m+2$,
				and the closed subsets $Z_m^1$,
				$Z_m^2$ and $Z_m^3$ are irreducible of codimension $m+2$ in $(\mathbb{A}^3)_m$.
			\item[(b)]
			By Definition \ref{既約成分の定義イデアルD_4^0},
			we have
				\begin{alignat*}{3}
				Z_m^1 &= \overline{\mathbf{V}(J_m^1) \cap \mathbf{D}(z_1)},\\
				Z_m^2 &= \overline{\mathbf{V}(J_m^2) \cap \mathbf{D}(y_1)},\\
				Z_m^3 &= \overline{\mathbf{V}(J_m^3) \cap \mathbf{D}(y_1)}.
				\end{alignat*}
			Moreover,
			we have $y_1+z_1 \in J_m^3$,
			hence we have $I_m^3 = J_m^3\cdot(R_m)_{y_1}\cap R_m = J_m^3\cdot(R_m)_{z_1}\cap R_m$ and therefore
				\begin{alignat*}{3}
				Z_m^3 &= \overline{\mathbf{V}(J_m^3) \cap \mathbf{D}(z_1)}.
				\end{alignat*}
			\item[(c)]
				For $m \geq 4$,
				we have $f^{(4)}_{222} = x_2^2$ by Lemma \ref{reduction of f^{(l)} g^{(l)} modulo L_{pqr}}(b).
				Hence
					\begin{center}
						$Z_m^0 = \mathbf{V}(L_{322} + \langle f^{(0)},...,f^{(m)} \rangle)$.
					\end{center}
		\end{itemize}
	\end{Rem}
Here,
we note that $Z_m^1$,
$Z_m^2$ and $Z_m^3$ have the following symmetries.
	\begin{Lemma}\label{symmetries of $S$}
	Assume $m \geq 3$.
	The closed subsets $Z_m^0$,
	$Z_m^1$,
	$Z_m^2$ and $Z_m^3$ are permutated by $\psi_1$ and $\psi_2$ (see Notation \ref{symmetric morphisms in D_4^0}) as follows.
		\begin{itemize}
		\item[(a)]
			$\psi_1(Z_m^0) = Z_m^0, \psi_1(Z_m^1) = Z_m^2, \psi_1(Z_m^2) = Z_m^1$ and $\psi_1(Z_m^3) = Z_m^3$.
		\item[(b)]
			$\psi_2(Z_m^0) = Z_m^0, \psi_2(Z_m^1) = Z_m^1, \psi_2(Z_m^2) = Z_m^3$ and $\psi_2(Z_m^3) = Z_m^2$.
		\end{itemize}
	\end{Lemma}
	\begin{proof}
	(a) We may think of $\varphi_1$ as an automorphism of the quotient field of $R_m$ which preserves $R_m$ and maps $(R_m)_{z_1}$ to $(R_m)_{y_1}$ and $(R_m)_{y_1}$ to $(R_m)_{z_1}$.
	Arguing as in Lemma \ref{symmetric morphisms}(b) and (c),
	we have
		\begin{alignat*}{7}
		\varphi_1(I_m^0) &= I_m^0,\\
		\varphi_1(I_m^1) &= \varphi_1(J_m^1\cdot (R_m)_{z_1} \cap R_m) &&= \varphi_1(J_m^1\cdot (R_m)_{z_1}) \cap R_m &&= J_m^2\cdot (R_m)_{y_1} \cap R_m &&= I_m^2,\\
		\varphi_1(I_m^2) &= \varphi_1(J_m^2\cdot (R_m)_{y_1} \cap R_m) &&= \varphi_1(J_m^2\cdot (R_m)_{y_1}) \cap R_m &&= J_m^1\cdot (R_m)_{z_1} \cap R_m &&= I_m^1,\\
		\varphi_1(I_m^3) &= \varphi_1(J_m^3\cdot (R_m)_{z_1} \cap R_m) &&= \varphi_1(J_m^3\cdot (R_m)_{z_1}) \cap R_m &&= J_m^3\cdot (R_m)_{y_1} \cap R_m &&= I_m^3.
		\end{alignat*}
	(see Remark \ref{the codimension of the irreducible components}(b)).
	Hence we have
		\begin{center}
		$\psi_1(Z_m^0) = Z_m^0, \psi_1(Z_m^1) = Z_m^2, \psi_1(Z_m^2) = Z_m^1$ and $\psi_1(Z_m^3) = Z_m^3$.
		\end{center}
	We can prove (b) in the same way as (a).
	\end{proof}
In the following,
we give the irreducible decomposition of the singular fiber $S_m^0$.
First of all,
we give the decomposition of $S_m^0$.
	\begin{Prop}\label{rough decomposition of D_4^0}
		For $m \geq 3$,
		we have
			\begin{align*}
				S_m^0	&= Z_m^0 \cup Z_m^1 \cup Z_m^2 \cup Z_m^3.
			\end{align*}
		Moreover,
		$Z_m^1$, $Z_m^2$ and $Z_m^3$ are pairwise distinct.
	\end{Prop}
	\begin{proof}
		By Lemma \ref{reduction of f^{(l)} g^{(l)} modulo L_{pqr}}(a) and (b),
		we have
			\begin{align*}
				f_{111}^{(0)}	&= f_{111}^{(1)} = 0,\\
				f^{(2)}_{111}	&= x_1^2.
			\end{align*}
		Hence the defining ideal of $S_m^0$ is
			\begin{center}
				$\sqrt{\langle x_0,y_0,z_0,x_1^2,f^{(3)},...,f^{(m)} \rangle} = \sqrt{\langle x_0,y_0,z_0,x_1,f^{(3)}_{211},...,f^{(m)}_{211} \rangle}$.
			\end{center}
		Using Lemma \ref{reduction of f^{(l)} g^{(l)} modulo L_{pqr}}(d),
		we have
			\begin{center}
				$f_{211}^{(3)} = y_1z_1(y_1+z_1)$.
			\end{center}
		Thus we have
			\begin{align*}
				S_m^0	=& \mathbf{V}(\langle x_0,x_1,y_0,y_1,z_0,f^{(4)},...,f^{(m)} \rangle) \cup \mathbf{V}(\langle x_0,x_1,y_0,z_0,z_1,f^{(4)},...,f^{(m)} \rangle)\\
						 & \cup \mathbf{V}(\langle x_0,x_1,y_0,z_0,y_1+z_1,f^{(4)},...,f^{(m)} \rangle)\\
						=& \mathbf{V}(J_m^1) \cup \mathbf{V}(J_m^2) \cup \mathbf{V}(J_m^3).
			\end{align*}
		Since $(\mathbb{A}^3)_m = \mathbf{D}(y_1) \cup \mathbf{D}(z_1) \cup \mathbf{V}(y_1,z_1)$ and $S_m^0$ is closed,
		we have
			\begin{center}
				$S_m^0 = S_m^0 \cap (\mathbb{A}^3)_m = \overline{S_m^0 \cap \mathbf{D}(y_1)} \cup \overline{S_m^0 \cap \mathbf{D}(z_1)} \cup (S_m^0 \cap \mathbf{V}(y_1,z_1))$.
			\end{center}
		Note that $\mathbf{V}(J_m^1) \cap \mathbf{D}(y_1) = \emptyset$ since $y_1 \in J_m^1$.
		Thus
			\begin{center}
				$S_m^0 \cap \mathbf{D}(y_1) = (\mathbf{V}(J_m^2) \cap \mathbf{D}(y_1)) \cup (\mathbf{V}(J_m^3) \cap \mathbf{D}(y_1))$.
			\end{center}
		Similarly,
			\begin{center}
				$S_m^0 \cap \mathbf{D}(z_1) = (\mathbf{V}(J_m^1) \cap \mathbf{D}(z_1)) \cup (\mathbf{V}(J_m^3) \cap \mathbf{D}(z_1))$.
			\end{center}
		By Remark \ref{the codimension of the irreducible components},
		we have
			\begin{center}
				$\overline{S_m^0 \cap \mathbf{D}(y_1)} = Z_m^2 \cup Z_m^3$
			\end{center}
		and
			\begin{center}
				$\overline{S_m^0 \cap \mathbf{D}(z_1)} = Z_m^1 \cup Z_m^3$.
			\end{center}
		Clearly,
		we have $S_m^0 \cap \mathbf{V}(y_1,z_1) = \mathbf{V}(I_m^0) = Z_m^0$.
		Therefore,
			\begin{align*}
				S_m^0	&=	\overline{\mathbf{V}(J_m^1) \cap \mathbf{D}(z_1)} \cup \overline{\mathbf{V}(J_m^2) \cap \mathbf{D}(y_1)} \cup \overline{\mathbf{V}(J_m^3) \cap \mathbf{D}(y_1)} \cup (S_m^0 \cap \mathbf{V}(y_1,z_1))\\
						&=	Z_m^0 \cup Z_m^1 \cup Z_m^2 \cup Z_m^3.
			\end{align*}
		Finally,
		we check that $Z_m^1$, $Z_m^2$ and $Z_m^3$ are pairwise distinct.
		By Lemma \ref{symmetries of $S$},
		it suffices to show that $Z_m^3 \not\subseteq Z_m^1$.
		The jet $P=(0,t,t)$ corresponds to the point where $y_1 = z_1 = 1$ and all of the other coordinates are $0$,
		and we see $P \in Z_m^3$ and $P \not\in \mathbf{V}(J_m^1) \supseteq Z_m^1$.
		Hence $Z_m^3 \not\subseteq Z_m^1$.
	\end{proof}
Next,
we give the irreducible decomposition of the singular fiber for small $m$.
	\begin{Prop}\label{小さな次数でのD_4^0の既約分解}
		For $0 \leq m \leq 4$,
		the irreducible decomposition of the singular fiber $S_m^0$ is as follows.
			\begin{enumerate}
				\item[(a)]
					$S_0^0 = \mathbf{V}(L_{111})$,
				\item[(b)]
					$S_1^0 = \mathbf{V}(L_{111})$,
				\item[(c)]
					$S_2^0 = \mathbf{V}(L_{211})$,
				\item[(d)]
					$S_3^0 = Z_3^1 \cup Z_3^2 \cup Z_3^3$,
				\item[(e)]
					$S_4^0 = Z_4^1 \cup Z_4^2 \cup Z_4^3$.
			\end{enumerate}
		Moreover,
		for $1 \leq m \leq 4$,
		the codimension of any irreducible component $Z \subseteq S_m^0$ is
			\begin{center}
				$\mathrm{codim}_{(\mathbb{A}^3)_m}\ Z = m+2$.
			\end{center}
	\end{Prop}
\begin{proof}
	(a)
	By Lemma \ref{reduction of f^{(l)} g^{(l)} modulo L_{pqr}}(a),
	$f^{(0)}_{111} = 0$,
	so $S_0^0 = \mathbf{V}(L_{111})$.
	\medskip\\
	(b)
	By Lemma \ref{reduction of f^{(l)} g^{(l)} modulo L_{pqr}}(a),
	$f^{(1)}_{111} = 0$,
	so $S_1^0 = \mathbf{V}(L_{111})$.
	\medskip\\
	(c)
	By Lemma \ref{reduction of f^{(l)} g^{(l)} modulo L_{pqr}}(b),
	$f^{(2)}_{111} = x_1^2$,
	so $S_2^0 = \mathbf{V}(L_{211})$.
	\medskip\\
	(d)
	By Proposition \ref{rough decomposition of D_4^0} and Remark \ref{the codimension of the irreducible components}(a),
	$S_3^0 = Z_3^0 \cup Z_3^1 \cup Z_3^2 \cup Z_3^3$ and $Z_3^1$, $Z_3^2$ and $Z_3^3$ are irreducible and pariwise distinct.
	Note that $J_3^1 = L_{221}$ is prime,
	so $I_3^1 = J_3^1 \subsetneq L_{222} = I_3^0$.
	Hence $Z_3^1 \supsetneq Z_3^0$.
	Thus the irreducible decomposition of $S_3^0 = Z_3^1 \cup Z_3^2 \cup Z_3^3$.
	\medskip\\
	(e)
	By Proposition \ref{rough decomposition of D_4^0},
		\begin{alignat*}{5}
			S_4^0	= Z_4^0 \cup Z_4^1 \cup Z_4^2 \cup Z_4^3.
		\end{alignat*}
	Here,
	$Z_4^0 = S_4^0 \cap \mathbf{V}(y_1,z_1)$ is not an irreducible component.
	Indeed,
	we have
		\begin{center}
			$Z_4^0 = \mathbf{V}(L_{222} + \langle f^{(4)}_{222} \rangle) = \mathbf{V}(L_{322})$
		\end{center}
	by Remark \ref{the codimension of the irreducible components}(c) and Lemma \ref{reduction of f^{(l)} g^{(l)} modulo L_{pqr}}(b),
	and hence
		\begin{center}
			$\mathrm{codim}_{(\mathbf{A}^3)_4}\ \mathbf{V}(L_{322}) = 3+2+2 = 7$,
		\end{center}
	while by Lemma \ref{codimension conditions}(b),
	the codimension of the irreducible component of $S_4^0$ is at most $4+2 = 6$.
	Thus,
	$S_4^0 \cap \mathbf{V}(y_1,z_1)$ is not an irreducible component of $S_4^0$ and,
	by Remark \ref{the codimension of the irreducible components}(a) and Proposition \ref{rough decomposition of D_4^0},
	the irreducible decomposition of $S_4^0$ is given by
		\begin{center}
			$Z_4^1 \cup Z_4^2 \cup Z_4^3$.
		\end{center}
\end{proof}
Next lemma is one key point to prove that $Z_m^0$ is irreducible for $m \geq 5$.
	\begin{Lemma} \label{D_4^0におけるZ_m^0の形}
		\begin{itemize}
			\item[(a)]
				For $m = 5$,
				$Z_5^0 = \mathbf{V}(L_{322})$.
			\item[(b)]
				For $m \geq 6$,
					\begin{center}
						$Z_m^0 \cong \mathbb{A}^{11} \times S_{m-6}$.
					\end{center}
		\end{itemize}
	\end{Lemma}
\begin{proof}
First note that,
by Definition \ref{既約成分の定義イデアルD_4^0},
Remark \ref{座標関数を含むイデアルの簡略化} and Remark \ref{the codimension of the irreducible components}(c),
$Z_m^0$ is defined by
	\begin{center}
		$I_m^0 = L_{322} + \langle f^{(0)},...,f^{(m)} \rangle = L_{322} + \langle f_{322}^{(0)},...,f_{322}^{(m)} \rangle$.
	\end{center}
Moreover,
by Lemma \ref{reduction of f^{(l)} g^{(l)} modulo L_{pqr}}(a) and (b),
we have
	\begin{align*}
		f^{(0)}_{322}	&= f^{(1)}_{322} = \cdots = f^{(5)}_{322} = 0.
	\end{align*}
(a)
Clearly,
we have
$I_5^0 = L_{322}+\langle f_{322}^{(0)},...,f^{(5)}_{322} \rangle = L_{322}$.
Hence $Z_5^0 = \mathbf{V}(L_{322})$.
\medskip\\
(b)
We have
	\begin{center}
		$I_m^0 = L_{322} + \langle f_{322}^{(6)},...,f_{322}^{(m)} \rangle$
	\end{center}
for $m \geq 6$.
As in Notation \ref{modulo polynomials},
we write $\mathbf{x}_3 = x_3t^3 + x_4t^4 + \cdots + x_mt^m$,
$\mathbf{y}_2 = y_2t^2 + y_3t^3 + \cdots + y_mt^m$ and $\mathbf{z}_2 = z_2t^2 + z_3t^3 + \cdots + z_mt^m$ and calculate $f(\mathbf{x}_3,\mathbf{y}_2,\mathbf{z}_2)$;
	\begin{alignat*}{5}
		f\left(\sum_{i=3}^mx_it^i,\sum_{i=2}^my_it^i,\sum_{i=2}^mz_it^i\right)
																		&= f\left(t^3\sum_{i=3}^mx_it^{i-3},t^2\sum_{i=2}^my_it^{i-2},
																		t^2\sum_{i=2}^mz_it^{i-2}\right)\\
																		&= t^6f\left(\sum_{i=3}^mx_it^{i-3},\sum_{i=2}^my_it^{i-2},\sum_{i=2}^mz_it^{i-2}\right)\\
																		&= t^6\big(f^{(0)}(x_3,y_2,z_2)+f^{(1)}(x_3,x_4,y_2,y_3,z_2,z_3)t+ \cdots\\
																		&\ \ \ +f^{(m-6)}(x_3,...,x_{m-3},y_2,...,y_{m-4},z_2,...,z_{m-4})t^{m-6} + \cdots \big).
		\end{alignat*}
(Note that the second equality holds because $f$ is weighted homogenous.)
We set $f^{(l)}_6 := f^{(l-6)}(x_3,...,x_{l-3}$ $,y_2,...,y_{l-4},z_2,...,z_{l-4})$ for $6 \leq l \leq m$,
then $f^{(l)}_{322} = f^{(l)}_6$.
Hence
		\begin{alignat}{7}
		Z_m^0 &= \mathbf{V}(L_{322} + \langle f_{322}^{(6)},...,f_{322}^{(m)} \rangle)\\
									   &= \mathbf{V}(L_{322} + \langle f^{(6)}_6,...,f^{(m)}_6 \rangle)
									   &\ \cong \mathbb{A}^{11}\times S_{m-6}.\label{Z_m^0の斉次性を使う表記}
		\end{alignat}
\end{proof}
Now,
we prove that $Z_m^0$ is irreducible of dimension $2m+1$ for $m \geq 5$.
To prove this claim,
it suffices to show that $S_m$ is irreducible of dimension $2(m+1)$ for $m \geq 0$ by the previous lemma.
	\begin{Rem}
		In \cite[Corollary 10.2.9]{Is},
		the following statement is proven:
		For a locally complete intersection variety $X$ and every positive integer $m$,
		the $m$-th jet scheme $X_m$ of $X$ is irreducible if and only if $X$ has Mather-Jacobian canonical singularities.
		
		We note that if $X$ is a normal locally complete intersection variety,
		then the notion of Mather-Jacobian canonical singularities coincides with the notion of canonical singularities by Proposition 10.1.10 in \cite{Is}.
		Hence a singular surface of type $D_4^0$ is Mather-Jacobian canonical,
		and the above result can be applied.
		
		For the reader's convenience,
		we prove the irreducibility of $S_m$ by a direct calculation.
	\end{Rem}
	\begin{Prop}\label{irreducibility of Z_m^0 and S_m of D_4^0}
		\begin{itemize}
			\item[(a)]
				For $m \geq 0$,
				$S_m$ is irreducible of dimension $2(m+1)$ (or equivalently of codimension $m+1$ in $(\mathbb{A}^3)_m$).
			\item[(b)]
				For $m \geq 5$,
				$Z_m^0$ is irreducible of dimension $2m+1$ (or equivalently of codimension $m+2$ in $(\mathbb{A}^3)_m$).
		\end{itemize}
	\end{Prop}
\begin{proof}
	By Remark \ref{非特異部分の逆像},
		\begin{center}
			$S_m = \overline{(\pi_m^{S})^{-1}(S_{\mathrm{sm}})} \cup S_m^0$
		\end{center}
	and $\overline{(\pi_m^{S})^{-1}(S_{\mathrm{sm}})}$ is irreducible of dimension $2(m+1)$.
	By Lemma \ref{codimension conditions}(a), 
	in order to show the assertion (a) for $S_m$,
	it suffices to show that $\mathrm{dim}\ S_m^0 \leq 2m+1$.\medskip\\
	For $0 \leq m \leq 4$,
	this holds by Proposition \ref{小さな次数でのD_4^0の既約分解}.
	Thus (a) holds for $0 \leq m \leq 4$.
	
	Before proving the statements for $m \geq 5$,
	we note that,
	by Proposition \ref{rough decomposition of D_4^0},
	we have
		\begin{alignat}{5}\label{decomp of S_m^0 D_4^0}
			S_m^0	&= Z_m^0 \cup Z_m^1 \cup Z_m^2 \cup Z_m^3,
		\end{alignat}
	and,
	by Remark \ref{the codimension of the irreducible components},
	$Z_m^1$,
	$Z_m^2$ and $Z_m^3$ are irreducible of dimension $2m+1$.
	Thus it suffices to show that the remaining part of $S_m^0$,
	i.e. $Z_m^0$,
	is irreducible of dimension $2m+1$ for $m \geq 5$,
	i.e. the statement (b).
	
	For $m = 5$,
	we have $Z_5^0 = \mathbf{V}(L_{322})$ and $Z_5^0$ is irreducible of dimension $3(5+1)-7 = 11 (= 2\times5+1)$ by Lemma \ref{D_4^0におけるZ_m^0の形}(a).
	Thus $S_5$ is irreducible of dimension $12$.
	
	For $m \geq 6$,
	we assume that (a) holds for $S_{m'}$ with $0 \leq m' < m$ and show the assertion for $S_m$ and $Z_m^0$.
	By Lemma \ref{D_4^0におけるZ_m^0の形}(b),
	we have $Z_m^0 \cong \mathbb{A}^{11} \times S_{m-6}$.
	By the inductive hypothesis,
	$S_{m-6}$ is irreducible of dimension $2(m-6+1) = 2m-10$.
	Hence $Z_m^0$ is irreducible of dimension
		\begin{center}
			$\mathrm{dim}\ Z_m^0 = \mathrm{dim}\ \mathbb{A}^{11} \times S_{m-6} = 11+ (2m-10) = 2m+1$.
		\end{center}
	Thus $S_m$ is irreducible of dimension $2(m+1)$.
\end{proof}

Now,
we give the irreducible decomposition of $S_m^0$ for $m \geq 5$.
In the following theorem,
note that the number of irreducible components of the singular fiber is constant.
	\begin{Thm}\label{irreducible decomposition D_4^0}
	For $m \geq 5$,
	the irreducible decomposition of the singular fiber $S_m^0$ is given by
		\begin{center}
		$S_m^0 = Z_m^0 \cup Z_m^1 \cup Z_m^2 \cup Z_m^3$.
		\end{center}
	\end{Thm}
	\begin{proof}
	First of all,
	we note that,
	for $0 \leq i \leq 3$,
	the closed subsets $Z_m^i$ are irreducible of dimension $2m+1$ by Remark \ref{the codimension of the irreducible components} and Proposition \ref{irreducibility of Z_m^0 and S_m of D_4^0}(b).
	Moreover,
	by Proposition \ref{rough decomposition of D_4^0},
	we have
		\begin{alignat*}{7}
			S_m^0	&= Z_m^0 \cup Z_m^1 \cup Z_m^2 \cup Z_m^3
		\end{alignat*}
	for $m \geq 5$.
	Hence all that remains is to check that $Z_m^i$ are pairwise distinct for $0 \leq i \leq 3$.
	For $m \geq 5$,
	we have $Z_m^0 \cap (\mathbf{D}(y_1) \cup \mathbf{D}(z_1)) = \emptyset$ while $Z_m^i \cap (\mathbf{D}(y_1) \cup \mathbf{D}(z_1)) \neq \emptyset$ for $1 \leq  i \leq 3$,
	so $Z_m^0$ is different from $Z_m^1$,
	$Z_m^2$ and $Z_m^3$.
	Moreover,
	by Proposition \ref{rough decomposition of D_4^0},
	$Z_m^1$,
	$Z_m^2$ and $Z_m^3$ are pairwise distinct.
	This completes the proof.
	\end{proof}

In the following,
we determine the inclusion relation between the intersections of two irreducible components of the singular fiber of a singular surface of type $D_4^0$.
	\begin{Thm}\label{main theorem of D_4^0}
	Let $k$ be an algebraically closed field of characteristic $2$,
	$S \subset \mathbb{A}^3$ the surface defined by $f = x^2+y^2z+yz^2$ in the affine space over $k$,
	$S_m^0$ the singular fiber of the $m$-th jet scheme $S_m$ with $m \geq 5$ and $Z_m^0,...,Z_m^3$ its irreducible components as in Definition \ref{既約成分の定義イデアルD_4^0}.
		\begin{itemize}
			\item[(a)]
				For $0 \leq i < j \leq 3$,
				$Z_m^i \cap Z_m^j \subsetneq Z_m^0$.
			\item[(b)]
				For $1 \leq i, j \leq 3$ with $i \neq j$,
				$Z_m^0 \cap Z_m^i \not\subseteq Z_m^0 \cap Z_m^j$.
			\item[(c)]
				For $1 \leq i, j \leq 3$ with $i \neq j$,
				$Z_m^i \cap Z_m^j \subsetneq Z_m^0 \cap Z_m^i$.
			\item[(d)]
				For $1 \leq i < j \leq 3$ and $1 \leq l \leq 3$,
				$Z_m^0 \cap Z_m^l \not\subseteq Z_m^i \cap Z_m^j$.
		\end{itemize}
	In particular,
	for $0 \leq i<j \leq 3$,
	$Z_m^i \cap Z_m^j$ is maximal in  $\{Z_m^i \cap Z_m^j \mid i, j \in \{0,1,2,3\}, i \neq j\}$ with respect to the inclusion relation if and only if $(i,j) = (0,1), (0,2), (0,3)$.
	\end{Thm}
	\begin{Rem}
	Over algebraically closed fields of characteristic $0$,
	the same statment was proved in \cite[Theorem 3.17]{Ky}.
	\end{Rem}
	\begin{proof}
	(a)
	By looking at the dimensions,
	we have $Z_m^0 \cap Z_m^j \subsetneq Z_m^0$ for $j = 1,2,3$.
	Hence the assertion is holds if $i = 0$,
	so we may assume that $(i,j) = (1,2)$ by Lemma \ref{symmetries of $S$}.
	From the definitions of $Z_m^i$ (Definition \ref{既約成分の定義イデアルD_4^0}),
	it suffices to check that $\sqrt{I_m^1 + I_m^2} \supseteq I_m^0$.
	By the definitions of the ideals $J_m^1$ and $J_m^2$,
	we have $y_1 \in J_m^1$ and $z_1 \in J_m^2$,
	hence $\sqrt{J_m^1 + J_m^2} \supseteq L_{222}$.
	Therefore $\sqrt{I_m^1 + I_m^2} \supseteq \sqrt{J_m^1 + J_m^2} \supseteq I_m^0$.
	By looking at the dimensions,
	we have $Z_m^1 \cap Z_m^2 \subsetneq Z_m^0$.
	\medskip\\
	(b)
	By Lemma \ref{symmetries of $S$},
	we only have to show that $Z_m^0 \cap Z_m^1 \not\subseteq Z_m^0 \cap Z_m^2$.
	Let us consider the jet $\gamma = (0, 0, t^2) \in S_m$.
	We prove the following two claims:
		\begin{itemize}
		\item[$\rm(\hspace{.18em}i\hspace{.18em})$]
			$\gamma \in Z_m^0 \cap Z_m^1$,
		\item[$\rm(\hspace{.08em}ii\hspace{.08em})$]
			$\gamma \not\in Z_m^0 \cap Z_m^2$.
		\end{itemize}
	$\rm(\hspace{.18em}i\hspace{.18em})$
	We can easily check that $\gamma \in Z_m^0$ by Definition \ref{既約成分の定義イデアルD_4^0}.
	Let us consider the family
		\begin{center}
		$\gamma_s := (0,0,st+t^2)$ for $s \neq 0$.
		\end{center}
	We have $\gamma_s \in \mathbf{V}(J_m^1)$ and $\gamma_s \in \mathbf{D}(z_1)$,
	so $\gamma_s \in \overline{\mathbf{V}(J_m^1) \cap \mathbf{D}(z_1)} = Z_m^1$.
	When $s$ goes to $0$,
	$\gamma_s$ approaches $\gamma$ and this jet belongs to $Z_m^1$.
	Hence $\gamma \in Z_m^0 \cap Z_m^1$.
	\medskip\\
	$\rm(\hspace{.08em}ii\hspace{.08em})$
	We show that $z_2 \in \sqrt{I_m^0 + I_m^2}$,
	and then $\gamma \not\in Z_m^0 \cap Z_m^2$.
	By Lemma \ref{reduction of f^{(l)} g^{(l)} modulo L_{pqr}},
		\begin{center}
		$f^{(5)}_{212} = y_1^2z_3 + y_1z_2^2 = y_1(y_1z_3 + z_2^2)$.
		\end{center}
	Then we have $y_1z_3 + z_2^2 \in J_m^2\cdot (R_m)_{y_1} \cap R_m = I_m^2$,
	so
		\begin{center}
		$z_2^2 = -y_1z_3 + (y_1z_3 + z_2^2) \in I_m^0 + I_m^2$.
		\end{center}
	Hence we have $z_2 \in \sqrt{I_m^0 + I_m^2}$ and $\gamma \not\in Z_m^0 \cap Z_m^2$.
	Therefore,
	we have $Z_m^0 \cap Z_m^i \not\subseteq Z_m^0 \cap Z_m^j$,
	or equivalently $Z_m^0 \cap Z_m^i \not\subseteq Z_m^j$,
	for $i,j \in \{1,2,3\}$ with $i \neq j$.
	\medskip\\
	(c)
	By the assertion (a),
	we have $Z_m^i \cap Z_m^j \subseteq Z_m^0 \cap Z_m^i$ for $i,j \in \{ 1,2,3 \}$ with $i \neq j$.
	If $Z_m^i \cap Z_m^j = Z_m^0 \cap Z_m^i$,
	then
		\begin{center}
			$Z_m^0 \cap Z_m^i = Z_m^0 \cap Z_m^i \cap Z_m^j \subseteq Z_m^0 \cap Z_m^j$,
		\end{center}
	a contradiction to (b).
	Thus,
	$Z_m^i \cap Z_m^j \subsetneq Z_m^0 \cap Z_m^i$
	\medskip\\
	(d)
	Take any $i,j,l \in \{1,2,3\}$ with $i \neq j$ and $j \neq l$ (not necessarily $i \neq l$).
	If $Z_m^0 \cap Z_m^l \subseteq Z_m^i \cap Z_m^j$,
	then 
		\begin{center}
			$Z_m^0 \cap Z_m^l \subseteq Z_m^0 \cap Z_m^i \cap Z_m^j \subseteq Z_m^0 \cap Z_m^j$,
		\end{center}
	a contradiction to (b).
	Hence $Z_m^0 \cap Z_m^l \not\subseteq Z_m^i \cap Z_m^j$.
	\end{proof}
	
Now we define a graph $\Gamma$ from the information of $S_m^0$ as follows.
	\begin{Const}[{\cite[Construction 2.15]{Ky}}]\label{graph constructed by the singular fiber}
	The graph $\Gamma(S_m^0) = (V,E)$ is constructed as follows.
		\begin{itemize}
		\item $V$ is the set of irreducible components of $S_m^0$.
		\item $E$ is the set of all maximal elements of $\{Z_m^i \cap Z_m^j \mid i, j \in \{0,1,2,3\}, i \neq j\}$,
				and $Z_m^i \cap Z_m^j \in E$ connects $Z_m^i$ and $Z_m^j$.
		\end{itemize}
	In other words,
	there is given an edge between $Z_m^i$ and $Z_m^j$ if and only if $Z_m^i \cap Z_m^j$ is maximal among the intersections of two distinct irreducible components.
	\end{Const}
	\begin{Coro}\label{graph of D_4^0}
	For a singular surface $S$ of type $D_4^0$ and $m \geq 5$,
	the graph $\Gamma(S_m^0)$ obtained by Construction \ref{graph constructed by the singular fiber} is isomorphic to the resolution graph of $S$.
	\end{Coro}
	\begin{proof}
	By Theorem \ref{irreducible decomposition D_4^0} and Theorem \ref{main theorem of D_4^0},
	$\Gamma(S_m^0)$ is the pair of $V = \{ Z_m^0, Z_m^1, Z_m^2, Z_m^3 \}$ and $E = \{ Z_m^0 \cap Z_m^1, Z_m^0 \cap Z_m^2, Z_m^0 \cap Z_m^3 \}$.
	Hence $\Gamma(S_m^0)$ can be described as
		
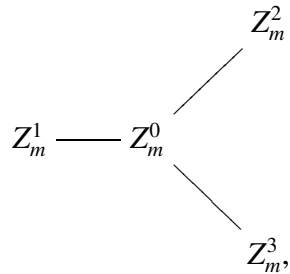
\[
			\xymatrix{
				 &  & Z_m^2\\
				Z_m^1 \ar@{-}[r] & Z_m^0 \ar@{-}[ru] \ar@{-}[rd] & \\
				 & & Z_m^3,
			}
		\]
	which is a Dynkin diagram of type $D_4$.
	\end{proof}
\section{Jet schemes of a singular surface of type $D_4^1$}
In this section,
we prove the statements of Theorem \ref{main theorem1} and Theorem \ref{main theorem2} on the singular surface of type $D_4^1$.
The proof goes mostly in the same way as in \cite{M2} and Theorem \ref{irreducible decomposition D_4^0},
but we need more elaborate analysis to show that $Z_m^0$ is irreducible.

Let $g = x^2+y^2z + yz^2 + xyz$ and $S = \mathbf{V}(g) \subseteq \mathbb{A}^3$,
which has a singularity of type $D_4^1$ at $0$ (\cite[Section 3]{Ar}).
Furthermore,
let $R_m = k[x_0,...,x_m,y_0,...,y_m,z_0,...,z_m]$,
$S_m$ the $m$-th jet scheme of $S$ and $S_m^0$ the singular fiber of $S_m$.
Here,
we note that the equation $g = x^2 + y^2z + yz^2 + xyz$ is not weighted homogenous,
and hence a certain part of the arguments of Mourtada \cite{M2} does not work in this case.
More specifically,
we used the fact that the equation is weighted homogenous to prove Lemma \ref{D_4^0におけるZ_m^0の形}(b),
but in this case we cannot use this idea.
In this section,
we focus on certain subsets of $Z_m^0$ and their dimensions (or equivalently codimensions) to prove the irreducibility of $Z_m^0$.
\medskip

Let the ideal $L_{pqr} := \langle x_0,...,x_{p-1},y_0,...,y_{q-1},z_0,...,z_{r-1} \rangle \subseteq R_m = k[x_0,...,x_m,y_0,...,y_m,z_0,...,z_m]$ for $0 < p,q,r \leq m+1$ be as in the previous section.
	\begin{Def}
		For $m \geq 1$,
		we consider the following ideals in $R_m$:
			\begin{alignat*}{11}
				J_m^1 &= L_{211} &&+ \langle y_1 \rangle &&+ \langle g^{(0)},...,g^{(m)} \rangle &&= L_{221} + \langle g^{(0)},...,g^{(m)} \rangle,\\
				J_m^2 &= L_{211} &&+ \langle z_1 \rangle &&+ \langle g^{(0)},...,g^{(m)} \rangle &&= L_{212} + \langle g^{(0)},...,g^{(m)} \rangle,\\
				J_m^3 &= L_{211} &&+ \langle y_1+z_1 \rangle &&+ \langle g^{(0)},...,g^{(m)} \rangle.
			\end{alignat*}
	\end{Def}
By Remark \ref{relations between Gamma(X_m) and affine coordinate} and $J_m^i \supset L_{111}$ for $i = 1,2,3$,
these ideals include the defining ideal of the singular fiber,
and hence correspond to closed subsets in the singular fiber $S_m^0$.

We have the following symmetries:
	\begin{Nota}\label{def of symmetric morphism of D_4^1}
		Let $\varphi_1$ and $\varphi_2$ be the automorphisms of $R_m$ defined by
		\[
			\varphi_1 :
				\begin{cases}
					x_i \mapsto x_i, \\
					y_i \mapsto z_i, \\
					z_i \mapsto y_i,
				\end{cases}
			\varphi_2 :
				\begin{cases}
					x_i \mapsto x_i, \\
					y_i \mapsto y_i, \\
					z_i \mapsto x_i + y_i + z_i.
				\end{cases}
		\]
		These automorphisms $\varphi_1$ and $\varphi_2$ induce ring isomorphisms $(\varphi_1)_{z_1} : (R_m)_{z_1} \rightarrow (R_m)_{y_1}$ and $(\varphi_2)_{y_1} : (R_m)_{y_1} \rightarrow (R_m)_{y_1}$.
		We write the isomorphisms corresponding to $\varphi_1$, $\varphi_2$, $(\varphi_1)_{z_1}$ and $(\varphi_2)_{y_1}$ as $\psi_1$, $\psi_2$, $(\psi_1)_{z_1}$ and $(\psi_2)_{y_1}$,
		respectively.
		For simplicity,
		we write $(\varphi_1)_{z_1}$, $(\varphi_2)_{y_1}$, $(\psi_1)_{z_1}$ and $(\psi_2)_{y_1}$ as $\varphi_1$, $\varphi_2$, $\psi_1$ and $\psi_2$.
	\end{Nota}
	\begin{Lemma}\label{symmetric morphisms of D_4^1}
		\begin{enumerate}
		\item[(a)]
			For $m \geq 1$, $i \in \{ 0,...,m \}$ and $k = 1,2$,
			$\varphi_k$ preserve $g^{(i)}$ i.e.,
				\begin{center}
					$\varphi_k(g^{(i)}) = g^{(i)}$
				\end{center}
			in $R_m$.
			In particular,
			the morphisms $\psi_k$ preserve $S_m$.
		\item[(b)]
			For $m \geq 1$,
				\begin{center}
					$\varphi_1(J_m^1\cdot (R_m)_{z_1}) = J_m^2\cdot (R_m)_{y_1}$
				\end{center}
			and
				\begin{center}
					$\varphi_2(J_m^2\cdot (R_m)_{y_1}) = J_m^3\cdot (R_m)_{y_1}$.
				\end{center}
		\item[(c)]
			The morphisms $\varphi_1$, $\varphi_2$, $\psi_1$ and $\psi_2$ preserve the union,
			the intersection and the inclusion relations of sets.
		\end{enumerate}
	\end{Lemma}	
	\begin{proof}
		The proofs of the assertions (b) and (c) are the same as those of Lemma \ref{symmetric morphisms}(b) and (c),
		so we only check the assertion (a).

		Note that the automorphisms $\varphi_1$ and $\varphi_2$ are induced by the automorphisms of $k[x,y,z]$ defined by
			\[
				\overline{\varphi_1} =
				\begin{cases}
				x \mapsto x, \\
				y \mapsto z, \\
				z \mapsto y,
				\end{cases}
				\overline{\varphi_2} =
				\begin{cases}
				x \mapsto x, \\
				y \mapsto y, \\
				z \mapsto x + y + z,
				\end{cases}
			\]
		and $\overline{\varphi_1}(g) = x^2 + z^2y + zy^2 + xzy = g$
		and
			\begin{alignat*}{7}
				\overline{\varphi_2}(g)	&= x^2+y^2(x+y+z)+y(x+y+z)^2+xy(x+y+z)\\
										&= x^2+y(x+y+z)(y+(x+y+z)+x)\\
										&= x^2+y(x+y+z)z	&=g,
			\end{alignat*}
		where $y+(x+y+z)+x = 2x+2y+z = z$ since $k$ is a field of characteristic $2$.
		In the same way as in Lemma \ref{symmetric morphisms}(a),
		$\varphi_k$ preserves the polynomials $g^{(i)}$ ($i \in \{0,...,m\}$),
		and the induced automorphism $\psi_k$ preserves $S_m$.
	\end{proof}
Now,
we prove the following lemma.
This lemma is one key point to prove that $Z_m^i$'s are irreducible for $0 \leq i \leq 3$.
	\begin{Lemma}\label{open irreducible and codimension}
	For $m, p, q, r \in \mathbb{Z}_{>0}$ with $m \geq 2p$,
	the following hold.
		\begin{itemize}
		\item[(a)]
			If $p \geq q > r$ and $2p = q + 2r$,
			then $\overline{(S_m \cap \mathbf{V}(L_{pqr})) \cap\mathbf{D}(z_r)}$ is irreducible and has codimension $m-p+q+r+1 = m + p - r +1$ in $(\mathbb{A}^3)_m$.
		\item[(b)]
			If $p \geq r > q$ and $2p = 2q+r$,
			then $\overline{(S_m \cap \mathbf{V}(L_{pqr})) \cap\mathbf{D}(y_q)}$ is irreducible and has codimension $m-p+q+r+1 = m+p-q+1$ in $(\mathbb{A}^3)_m$.
		\item[(c)]
			If $p > q=r$ and $2p > 3q$,
			then
				\begin{alignat*}{5}
				\overline{(S_m \cap \mathbf{V}(L_{pqq} + \langle y_q+z_q \rangle)) \cap\mathbf{D}(y_q)}
				= \overline{(S_m \cap \mathbf{V}(L_{pqq} + \langle y_q+z_q \rangle)) \cap\mathbf{D}(z_q)}
				\end{alignat*}
			holds,
			and this set is irreducible of codimension $m + p - q +1$ in $(\mathbb{A}^3)_m$.
		\item[(d)]
			If $p > q=r$ and $2p = 3q$,
			then
				\begin{center}
				$\overline{(S_m \cap \mathbf{V}(L_{pqq})) \cap\mathbf{D}(y_q)} = \overline{(S_m \cap \mathbf{V}(L_{pqq})) \cap\mathbf{D}(z_q)}$
				\end{center}
			holds,
			and this set is irreducible of codimension $m-p+2q+1 = m+p-q+1$ in $(\mathbb{A}^3)_m$.
			If furthermore $m = 2p$,
			then
				\begin{center}
					$\overline{(S_m \cap \mathbf{V}(L_{pqq})) \cap\mathbf{D}(y_q)} = S_m \cap \mathbf{V}(L_{pqq}) = \mathbf{V}(L_{pqq} + \langle g_{pqq}^{(2p)} \rangle)$.
				\end{center}
		\end{itemize}
	\end{Lemma}
	\begin{proof}
	Before proving the lemma,
	note that
		\begin{align*}
			S_m \cap \mathbf{V}(L_{pqr}) = \mathbf{V}(L_{pqr}+\langle g^{(0)},...,g^{(m)} \rangle) = \mathbf{V}(L_{pqr} + \langle g_{pqr}^{(0)},...,g_{pqr}^{(m)} \rangle)
		\end{align*}
	by Remark \ref{座標関数を含むイデアルの簡略化}.
	\medskip\\
	(a)
	By Lemma \ref{reduction of f^{(l)} g^{(l)} modulo L_{pqr}}(e),
	for an integer $l$ satisfying $2p \leq l \leq m$,
	there exists a polynomial $h^{(l)} \in k[x_p,...,x_l,y_q,...,y_{l-2r-1},z_r,...,z_l]$ such that
		\begin{center}
			$g_{pqr}^{(l)} = y_{l-2r}z_{r}^2 + h^{(l)}$.
		\end{center}
	On the other hand,
	if $l < 2p$,
	then we have $l < 2p = q+2r < 2q+r$.
	Hence by Lemma \ref{reduction of f^{(l)} g^{(l)} modulo L_{pqr}}(a),
	$g_{pqr}^{(l)} = 0$ and $(S_m \cap \mathbf{V}(L_{pqr})) \cap \mathbf{D}(z_r)$ is defined by
		\begin{alignat}{7} \label{商空間上での定義イデアル}
			(L_{pqr} + \langle g^{(0)},...,g^{(m)} \rangle) \cdot (R_m)_{z_r}
			=& L_{pqr} \cdot (R_m)_{z_r} + \left\langle y_{2p-2r}+\frac{h^{(2p)}}{z_r^2}, y_{2p+1-2r}+\frac{h^{(2p+1)}}{z_r^2},..., y_{m-2r}+\frac{h^{(m)}}{z_r^2} \right\rangle
		\end{alignat}
	in $\mathbf{D}(z_r) \subset (\mathbb{A}^3)_m$.
	This ideal is prime of height $p+q+r+(m-2p+1) = m-p+q+r+1 = m+p-r+1$.
	Therefore,
	$\overline{(S_m \cap \mathbf{V}(L_{pqr})) \cap\mathbf{D}(z_r)}$ is irreducible of codimension $m+p-r+1$.
	\medskip\\
	(b)
	Using the automorphism $\varphi_1$,
	we see that (b) follows from (a).
	\medskip\\
	(c)
	First,
	we note that
	$y_q = z_q$ holds on $S_m \cap \mathbf{V}(L_{pqq} + \langle y_q+z_q \rangle)$,
	so
		\begin{center}
			$\gamma \in (S_m \cap \mathbf{V}(L_{pqq} + \langle y_q+z_q \rangle)) \cap\mathbf{D}(y_q) \Leftrightarrow$
			$\gamma \in (S_m \cap \mathbf{V}(L_{pqq} + \langle y_q+z_q \rangle)) \cap\mathbf{D}(z_q)$.
		\end{center}
	Hence we have
		\begin{align*}
			\overline{(S_m \cap \mathbf{V}(L_{pqq} + \langle y_q+z_q \rangle)) \cap\mathbf{D}(y_q)}
			&= \overline{(S_m \cap \mathbf{V}(L_{pqq} + \langle y_q+z_q \rangle)) \cap\mathbf{D}(z_q)}.
		\end{align*}
	
	Next,
	we prove that the closed subset $\overline{(S_m \cap \mathbf{V}(L_{pqq} + \langle y_q+z_q \rangle)) \cap\mathbf{D}(y_q)}$ is irreducible of codimension $m+p-q+1$.
	It is enough to show that $(L_{pqq} + \langle y_q+z_q \rangle + \langle g^{(0)},...,g^{(m)} \rangle)\cdot (R_m)_{y_q}$ is prime of height $m+p-q+1$.
	Let us look at $g_{pqq}^{(l)}$ for $l$ with $l \leq m$.
	If $l < 3q$,
	then $l < 3q < 2p$.
	So $g_{pqq}^{(l)} = 0$ by Lemma \ref{reduction of f^{(l)} g^{(l)} modulo L_{pqr}}(a).
	If $l = 3q$,
	then $g_{pqq}^{(3q)} = y_q^2z_q+y_qz_q^2$ by Lemma \ref{reduction of f^{(l)} g^{(l)} modulo L_{pqr}}(d),
	and $g_{pqq}^{(3q)} \in \langle y_q+z_q \rangle$.
	Finally,
	if $3q < l \leq m$,
	$T_z(g_{pqq}^{(l)}) = y_q^2z_{l-2q}$ by Lemma \ref{reduction of f^{(l)} g^{(l)} modulo L_{pqr}}(h).
	So there exists a polynomial $h^{(l)} \in k[x_p,...,x_m,y_q,...,y_m,z_q,...,z_{l-2q-1}]$ such that
		\begin{center}
			$g_{pqq}^{(l)} = y_q^2z_{l-2q} + h^{(l)}$.
		\end{center}
	Therefore,
	$(S_m \cap \mathbf{V}(L_{pqq} + \langle y_q+z_q \rangle)) \cap\mathbf{D}(y_q)$ is defined by
	\begin{alignat}{7}
			&	(L_{pqq} + \langle y_q+z_q \rangle + \langle g^{(0)},...,g^{(m)} \rangle) \cdot (R_m)_{y_q}\\
			=&	(L_{pqq} + \langle y_q+z_q \rangle) \cdot (R_m)_{y_q} + \left\langle z_{q+1}+\frac{h^{(3q+1)}}{y_q^2}, z_{q+2}+\frac{h^{(3q+2)}}{y_q^2},..., z_{m-2q}+\frac{h^{(m)}}{y_q^2} \right\rangle \label{open affine of y_q}
		\end{alignat}
	and this ideal is prime of height $p+q+q+1+(m-3q) = m+p-q+1$.
\medskip\\
	(d)
	First,
	we prove the following claim.
	\begin{claim}
	$\overline{(S_m \cap \mathbf{V}(L_{pqq})) \cap\mathbf{D}(y_q)}$ is irreducible of codimension $m+p-q+1$.
	\end{claim}
	We think of the defining ideal $L_{pqq} + \langle g^{(0)},...,g^{(m)} \rangle$ of $S_m \cap \mathbf{V}(L_{pqq})$ as $L_{pqq} + \langle g^{(0)},...,g^{(2p)} \rangle + \langle g^{(2p+1)},...,g^{(m)} \rangle$.
	By Lemma \ref{reduction of f^{(l)} g^{(l)} modulo L_{pqr}}(a) and Lemma \ref{reduction of f^{(l)} g^{(l)} modulo L_{pqr}}(c),
		\begin{center}
		$L_{pqq} + \langle g^{(0)},...,g^{(2p)} \rangle = L_{pqq} + \langle g^{(2p)}_{pqq} \rangle$,
		\end{center}
	with $g^{(2p)}_{pqq} = x_p^2+y_q^2z_q+y_qz_q^2$ and by Notation \ref{modulo polynomials},
		\begin{center}
		$\{ g_{pqq}^{(2p+1)},...,g_{pqq}^{(m)} \} \subset k[x_p,...,x_m,y_q,...,y_m,z_q,...,z_m]$.
		\end{center}
	Now,
	we consider the following residue ring:
		\begin{align*}
		Q	&= k[x_0,...,x_p,y_0,...,y_q,z_0,...,z_q]/\langle x_0,...,x_{p-1},y_0,...,y_{q-1},z_0,...,z_{q-1}, g_{pqq}^{(2p)} \rangle\\
			&\cong k[x_p,y_q,z_q]/\langle g_{pqq}^{(2p)} \rangle.
		\end{align*}
	Since $g_{pqq}^{(2p)}$ is irreducible in $k[x_p,y_q,z_q]$ by Remark \ref{remark of the difference between f and g},
	this ring is an integral domain.
	Since
		\begin{alignat*}{5}
			Q_m 	&:= Q[x_{p+1},...,x_m,y_{q+1},...,y_m,z_{q+1},...,z_m]\\
					&\cong R_m/(\langle x_0,...,x_{p-1},y_0,...,y_{q-1},z_0,...,z_{q-1}, g_{pqq}^{(2p)} \rangle),
		\end{alignat*}
	$L_{pqq} + \langle g_{pqq}^{(2p)} \rangle$ is prime of height $p+q+q+1 = 3p-q+1$ and the claim holds if $m = 2p$.
	Moreover,
	we have $y_q \not\in \langle g_{pqq}^{(2p)} \rangle$ and $L_{pqq} + \langle g_{pqq}^{(2p)} \rangle$ is prime,
	so $\overline{(S_m \cap \mathbf{V}(L_{pqq})) \cap\mathbf{D}(y_q)} = S_m \cap \mathbf{V}(L_{pqq})$.
	
	In the following,
	we assume $m > 2p$.
	Let $\overline{g_{pqq}^{(2p+1)}},...,\overline{g_{pqq}^{(m)}}$ and $\overline{y_q}$ be the images of $g_{pqq}^{(2p+1)},...,g_{pqq}^{(m)}$ and $y_q$ by the natural surjection $R_m \rightarrow Q_m$.
	What we have to show is that $\langle \overline{g_{pqq}^{(2p+1)}},...,\overline{g_{pqq}^{(m)}} \rangle \cdot (Q_m)_{\overline{y_q}}$ is prime of height $m-2p$.
	Here,
	we note that $\overline{y_q} \neq 0$ in $Q_m$ since $y_q \not\in \langle g_{pqq}^{(2p)} \rangle$ and that $Q_{\overline{y_q}}$ is also an integral domain.
	Assume $2p+1 \leq l \leq m$.
	Then we have $l \geq 2p+1 > 2p = 3q$.
	Hence there exists $h^{(l)} \in Q[x_{p+1},...,x_l,y_{q+1},...,y_l,z_{q+1},...,z_{l-2q-1}]$ such that $\overline{g_{pqq}^{(l)}} = \overline{y_q}^2z_{l-2q} + h^{(l)}$ by Lemma \ref{reduction of f^{(l)} g^{(l)} modulo L_{pqr}}(h).
	Therefore,
	from
		\begin{align*}
			\langle \overline{g_{pqq}^{(2p+1)}},...,\overline{g_{pqq}^{(m)}} \rangle \cdot (Q_m)_{\overline{y_q}} 
			&= \left\langle z_{2p-2q+1}+\frac{h^{(2p+1)}}{\overline{y_q}^2}, z_{2p-2q+2}+\frac{h^{(2p+2)}}{\overline{y_q}^2},..., z_{m-2q}+\frac{h^{(m)}}{\overline{y_q}^2} \right\rangle,
		\end{align*}
	it follows that $\langle \overline{g_{pqq}^{(2p+1)}},...,\overline{g_{pqq}^{(m)}} \rangle \cdot (Q_m)_{\overline{y_q}}$ is prime of height $m-2p$ in $(Q_m)_{\overline{y_q}} \cong Q_{\overline{y_q}}[x_{p+1},...,x_m,$ $y_{q+1},...,y_m,z_{q+1},...,z_m]$.
	This proves the claim.
	
	Finally,
	we prove
		\begin{center}
		$\overline{(S_m \cap \mathbf{V}(L_{pqq})) \cap\mathbf{D}(y_q)} = \overline{(S_m \cap \mathbf{V}(L_{pqq})) \cap\mathbf{D}(z_q)}$.
		\end{center}
	Note that the right hand side is also irreducible by the symmetry.
	From $g(0,t^q,t^q) = 0$,
	it follows that $(0,t^q,t^q) \in (S_m \cap \mathbf{V}(L_{pqq})) \cap\mathbf{D}(y_qz_q)$.
	Since an irreducible closed set is the closure of its nonempty open subset,
	we have
		\begin{align*}
			\overline{(S_m \cap \mathbf{V}(L_{pqq})) \cap\mathbf{D}(y_q)} &= \overline{(S_m \cap \mathbf{V}(L_{pqq})) \cap\mathbf{D}(y_qz_q)}
			= \overline{(S_m \cap \mathbf{V}(L_{pqq})) \cap\mathbf{D}(z_q)}.
		\end{align*}
	\end{proof}
	\begin{Coro}\label{calculation of the codimension of three irreducible components of S_m^0 of D_4^1}
	For $m \geq 3$,
	the ideals $J_m^1\cdot(R_m)_{z_1}$,
	$J_m^2\cdot(R_m)_{y_1}$ and $J_m^3\cdot(R_m)_{y_1}$ are prime,
	and the closed subsets $\overline{\mathbf{V}(J_m^1) \cap \mathbf{D}({z_1})}$,
	$\overline{\mathbf{V}(J_m^2) \cap \mathbf{D}({y_1})}$ and $\overline{\mathbf{V}(J_m^3) \cap \mathbf{D}({y_1})}$ are irreducible of dimension $2m+1$ (or equivalently of codimension $m+2$ in $(\mathbb{A}^3)_m$).
	\end{Coro}
	\begin{proof}
	For $m=3$,
	it is easy to see that $J_3^1$,
	$J_3^2$ and $J_3^3$ are prime of height $5$.
	For instance,
	by Lemma \ref{reduction of f^{(l)} g^{(l)} modulo L_{pqr}}(a),
	$f^{(0)}_{221} = f^{(1)}_{221} = f^{(2)}_{221} = f^{(3)}_{221} = 0$.
	Hence $J_3^1 = L_{221}$ is prime of height $5$.
	Similarly for $J_3^2$ and $J_3^3$.
	
	For $m \geq 4$,
	we apply Lemma \ref{open irreducible and codimension}(a) with $p = q = 2$ and $r=1$ to show that $\overline{\mathbf{V}(J_m^1) \cap \mathbf{D}({z_1})}$ is irreducible of codimension
		\begin{center}
			$m+2-1+1 = m+2$.
		\end{center}
	We see that $J_m^1\cdot(R_m)_{z_1}$ is prime from \eqref{商空間上での定義イデアル} in the proof of Lemma \ref{open irreducible and codimension}(a).
	
	Using Lemma \ref{symmetric morphisms of D_4^1}(b),
	we also see that $J_m^2\cdot(R_m)_{y_1}$ and $J_m^3\cdot(R_m)_{y_1}$ are prime
	and $\overline{\mathbf{V}(J_m^2) \cap \mathbf{D}({y_1})}$ and $\overline{\mathbf{V}(J_m^3) \cap \mathbf{D}({y_1})}$ are irreducible of codimension $m+2$.
	\end{proof}

Now we define ideals in $R_m$ that will turn out to be defining ideals of the irreducible components of $S_m^0$ for $m \geq 5$.
	\begin{Def} \label{definition of the irreducible components of singular fiber of D_4^1}
		For $m \geq 1$,
		we define
		\begin{alignat}{5}
			\label{define of the ideal of irreducible components of D_4^1 0}	I_m^0 &= L_{222} + \langle g^{(0)},...,g^{(m)} \rangle,\\
			\label{define of the ideal of irreducible components of D_4^1 1}	I_m^1 &= J_m^1\cdot(R_m)_{z_1} \cap R_m,\\
			\label{define of the ideal of irreducible components of D_4^1 2}	I_m^2 &= J_m^2\cdot(R_m)_{y_1} \cap R_m,\\
			\label{define of the ideal of irreducible components of D_4^1 3}	I_m^3 &= J_m^3\cdot(R_m)_{y_1} \cap R_m.
		\end{alignat}
	Furthermore,
	we define closed subsets
		\begin{center}
			$Z_m^i = \mathbf{V}(I_m^i)$
		\end{center}
		for $0 \leq i \leq 3$.
	\end{Def}
	\begin{Rem}\label{D_4^1におけるI_m^0の変形}
		By Lemma \ref{reduction of f^{(l)} g^{(l)} modulo L_{pqr}}(b),
		we have $f^{(4)}_{222} = x_2^2$.
		Hence,
		if $m \geq 4$,
			\begin{center}
				$Z_m^0 = \mathbf{V}(L_{322} + \langle g^{(0)},...,g^{(m)} \rangle) = S_m \cap \mathbf{V}(L_{322})$.
			\end{center}
		
		Assume $2p = 3q$ and $m \geq 2p$.
By Remark \ref{座標関数を含むイデアルの簡略化} and Lemma \ref{reduction of f^{(l)} g^{(l)} modulo L_{pqr}}(a),
	\begin{align}\label{property of G_{pqq}^{m}}
		L_{pqq}+\langle g^{(0)},...,g^{(m)} \rangle = L_{pqq}+\langle g_{pqq}^{(0)},...,g_{pqq}^{(m)} \rangle = L_{pqq}+\langle g_{pqq}^{(2p)},...,g_{pqq}^{(m)} \rangle.
	\end{align}
The same arguments show that
	\begin{align}\label{irreducible components 2p-1}
		L_{pqq}+\langle g^{(0)},...,g^{(2p-1)} \rangle = L_{pqq}.
	\end{align}
	\end{Rem}

	\begin{Lemma}\label{irreducibility of Z_m^0}
	For $m \geq 5$,
	$Z_m^0$ is irreducible of dimension $2m+1$ (or equivalently of codimension $m+2$ in $(\mathbb{A}^3)_m$).
	\end{Lemma}
The strategy of the proof of this lemma is as follows:
Let $u = \lfloor m/6 \rfloor$.
We divide $Z_m^0$ into $Z_m^0 \cap \Big(\mathbf{V}(L_{3,u',u'}) \cap \big(\mathbf{D}(y_{u'}) \cup \mathbf{D}(z_{u'})\big) \Big)$ for $u' < 2u$ and $Z_m^0 \cap \mathbf{V}(L_{3,2u,2u})$ and calculate their codimensions.
The codimensions of the former are calculated in the same way as in Lemma \ref{open irreducible and codimension}.
As for the latter,
we calculate these sets in Proposition \ref{center irreducible components of D_4^1}.
On the other hand,
we have an upper bound for the codimension of the irreducible components of $Z_m^0$. 
Then we can argue as in \cite{M1} to conclude that $Z_m^0$ is irreducible.

Before proving Lemma \ref{irreducibility of Z_m^0},
we show the following proposition.
	\begin{Prop}\label{center irreducible components of D_4^1}
	Let $p,q$ and $u$ be positive integers with $p = 3u$ and $q = 2u$,
	and assume $2p \leq m < 2(p+3)$.
	We set
		\begin{alignat*}{7}
			U_m	&:= \mathbf{V}(L_{pqq} + \langle g_{pqq}^{(2p)},...,g_{pqq}^{(m)} \rangle)	&&= S_m \cap \mathbf{V}(L_{pqq}),\\
			V_m	&:= \mathbf{V}(L_{p,q+1,q+1} + \langle g_{p,q+1,q+1}^{(2p)},...,g_{p,q+1,q+1}^{(m)} \rangle)	&&= S_m \cap \mathbf{V}(L_{p,q+1,q+1}),\\
			W_m	&:= \mathbf{V}(L_{p+2,q+2,q+2} + \langle g_{p+2,q+2,q+2}^{(2p)},...,g_{p+2,q+2,q+2}^{(m)} \rangle)	&&= S_m \cap \mathbf{V}(L_{p+2,q+2,q+2}).
		\end{alignat*}
	\begin{itemize}
		\item[(a)]
			The codimension of the closed subset $W_m$ in $(\mathbb{A}^3)_m$ is greater than $m+u+1$.
		\item[(b)]
			The codimension of the closed subset $V_m$ in $(\mathbb{A}^3)_m$ is greater than $m+u+1$.
		\item[(c)]
			$U_m$ is irreducible of codimension $m+u+1$ in $(\mathbb{A}^3)_m$,
	and
		\begin{align*}
			U_m = \overline{(S_m \cap \mathbf{V}(L_{pqq})) \cap \mathbf{D}(y_q)} =\overline{(S_m \cap \mathbf{V}(L_{pqq})) \cap \mathbf{D}(z_q)}.
		\end{align*}
	\end{itemize}
	\end{Prop}
	\begin{proof}
	We fix $u$ (and $p,q$),
	and deal with the cases $m = 2p, 2p+1,...,2p+5$ in order.
	\medskip\\
	(a)
	The cases $m \leq 2p+3$.
	By Lemma \ref{reduction of f^{(l)} g^{(l)} modulo L_{pqr}}(a),
	we have
		\begin{center}
			$g_{p+2,q+2,q+2}^{(2p)} = \cdots = g_{p+2,q+2,q+2}^{(2p+3)} = 0$.
		\end{center}
	Hence
		\begin{align*}
			L_{p+2,q+2,q+2} + \langle g^{(2p)}_{p+2,q+2,q+2},...,g^{(m)}_{p+2,q+2,q+2} \rangle
			&=	L_{p+2,q+2,q+2}
		\end{align*}
	and
		\begin{center}
			$\mathrm{codim}_{(\mathbb{A}^3)_m}\ W_m = p+2+q+2+q+2 = 7u+6 > 7u+4 \geq m+u+1$.
		\end{center}
		
	The case $m = 2p+4$.
	By Lemma \ref{reduction of f^{(l)} g^{(l)} modulo L_{pqr}}(b),
	$g_{p+2,q+2,q+2}^{(2p+4)} = x_{p+2}^2$.
	Hence
		\begin{center}
			$W_{2p+4} = \mathbf{V}(L_{p+2,q+2,q+2}+\langle x_{p+2}^2 \rangle) = \mathbf{V}(L_{p+3,q+2,q+2})$.
		\end{center}
	Thus,
	$W_{2p+4}$ is irreducible of codimension $p+3+q+2+q+2 = 7u+7 > 7u+5 = (2p+4)+u+1$ in $(\mathbb{A}^3)_m$.
	
	The case $m = 2p+5$.
	We note that
		\begin{center}
			$W_{2p+5} = (\pi_{2p+5,2p+4}^S)^{-1}(W_{2p+4}) = \mathbf{V}(L_{p+3,q+2,q+2} + \langle g_{p+3,q+2,q+2}^{(2p+5)} \rangle)$.
		\end{center}
	By Lemma \ref{reduction of f^{(l)} g^{(l)} modulo L_{pqr}}(a),
	$g_{p+3,q+2,q+2}^{(2p+5)} = 0$.
	Hence $W_{2p+5}$ is irreducible of codimension $p+3+q+2+q+2 = 7u+7 > 7u+6 = (2p+5)+u+1$ in $(\mathbb{A}^3)_m$.
	This completes the proof of (a).
	\medskip\\
	(b)
	The case $m = 2p$.
	By Lemma \ref{reduction of f^{(l)} g^{(l)} modulo L_{pqr}}(b),
	$g_{p,q+1,q+1}^{(2p)} = x_p^2$.
	Hence
		\begin{center}
			$V_{2p} = \mathbf{V}(L_{p,q+1,q+1} + \langle x_p^2 \rangle) = \mathbf{V}(L_{p+1,q+1,q+1})$
		\end{center}
	and $V_{2p}$ is of codimension $p+1+q+1+q+1 = 7u+3 > 7u+1 = 2p+u+1$.
	
	We note that,
	for $m \geq 2p+1$,
		\begin{center}
			$V_m = (\pi_{m,2p}^S)^{-1}(V_{2p}) = \mathbf{V}(L_{p+1,q+1,q+1}+\langle g_{p+1,q+1,q+1}^{(2p+1)},...,g_{p+1,q+1,q+1}^{(m)} \rangle)$.
		\end{center}

	The case $m = 2p+1$.
	By Lemma \ref{reduction of f^{(l)} g^{(l)} modulo L_{pqr}}(a),
	$g_{p+1,q+1,q+1}^{(2p+1)} = 0$ and $V_{2p+1}$ is of codimension
		\begin{center}
			$p+1+q+1+q+1 = 7u+3 > 7u+2 = (2p+1)+u+1$.
		\end{center}
	
	The case $m = 2p+2$.
	By Lemma \ref{reduction of f^{(l)} g^{(l)} modulo L_{pqr}}(b),
	$g_{p+1,q+1,q+1}^{(2p+2)} = x_{p+1}^2$.
	Hence
		\begin{center}
			$V_{2p+2} = \mathbf{V}(L_{p+1,q+1,q+1}+\langle x_{p+1}^2 \rangle) = \mathbf{V}(L_{p+2,q+1,q+1})$
		\end{center}
	and $V_{2p+2}$ is of codimension $p+2+q+1+q+1 = 7u+4 > 7u+3 = (2p+2)+u+1$.
	
	The case $m = 2p+3$.
	By Lemma \ref{reduction of f^{(l)} g^{(l)} modulo L_{pqr}}(d),
	$g_{p+2,q+1,q+1}^{(2p+3)} = y_{q+1}z_{q+1}(y_{q+1}+z_{q+1})$.
	Hence
		\begin{alignat*}{7}
			V_{2p+3}	&=	\mathbf{V}(L_{p+2,q+1,q+1} + \langle y_{q+1}z_{q+1}(y_{q+1}+z_{q+1}) \rangle)\\
						&=	\mathbf{V}(L_{p+2,q+2,q+1}) \cup \mathbf{V}(L_{p+2,q+1,q+2}) \cup \mathbf{V}(L_{p+2,q+1,q+1} + \langle y_{q+1}+z_{q+1} \rangle).
		\end{alignat*}
	Thus
		\begin{center}
			$\mathrm{codim}_{(\mathbb{A}^3)_m}\ V_{2p+3} = 7u+5 > 7u+4 = (2p+3)+u+1$.
		\end{center}
	
	The cases $m \geq 2p+4$.
	Since $V_m = (\pi_{m,2p+3}^S)^{-1}(V_{2p+3}) = (S_m \cap \mathbf{V}(L_{p+2,q+2,q+1})) \cup (S_m \cap \mathbf{V}(L_{p+2,q+1,q+2})) \cup (S_m \cap \mathbf{V}(L_{p+2,q+1,q+1} + \langle y_{q+1}+z_{q+1} \rangle))$ and $(\mathbb{A}^3)_m = \mathbf{D}(y_{q+1}) \cup \mathbf{D}(z_{q+1}) \cup \mathbf{V}(y_{q+1},z_{q+1})$,
	we have
		\begin{alignat*}{7}
			V_m	&=	\overline{(S_m \cap \mathbf{V}(L_{p+2,q+2,q+1})) \cap \mathbf{D}(z_{q+1})} \cup \overline{(S_m \cap \mathbf{V}(L_{p+2,q+1,q+2})) \cap \mathbf{D}(y_{q+1})}\ \cup\\
				&\ \ \ \ \overline{(S_m \cap \mathbf{V}(L_{p+2,q+1,q+1} + \langle y_{q+1}+z_{q+1} \rangle)) \cap \mathbf{D}(y_{q+1})}\ \cup \\
				&\ \ \ \ \overline{(S_m \cap \mathbf{V}(L_{p+2,q+1,q+1} + \langle y_{q+1}+z_{q+1} \rangle)) \cap \mathbf{D}(z_{q+1})} \cup (V_m\cap \mathbf{V}(y_{q+1},z_{q+1})).
		\end{alignat*}
	Note that $V_m\cap \mathbf{V}(y_{q+1},z_{q+1}) = W_m$.
	By Lemma \ref{open irreducible and codimension}(a),
	(b) and (c) and the statement (a) above,
		\begin{center}
			$\mathrm{codim}_{(\mathbb{A}^3)_m}\ V_m > m+u+1$.
		\end{center}
	This complete the proof of (b).
	\medskip\\
	(c)
	Since $(\mathbb{A}^3)_m = \mathbf{D}(y_{q}) \cup \mathbf{D}(z_{q}) \cup \mathbf{V}(y_{q},z_{q})$ and $U_m$ is closed,
	we have
		\begin{alignat*}{5}
			U_m	&= U_m \cap (\mathbb{A}^3)_m
				&&= \overline{U_m\cap \mathbf{D}(y_q)} \cup  \overline{U_m\cap \mathbf{D}(z_q)} \cup (U_m \cap \mathbf{V}(y_q,z_q)).
		\end{alignat*}
	By Lemma \ref{open irreducible and codimension}(d),
	we have
		\begin{center}
			$\overline{U_m\cap \mathbf{D}(y_q)} =  \overline{U_m\cap \mathbf{D}(z_q)}$
		\end{center}
	and $\overline{U_m\cap \mathbf{D}(y_q)}$ is irreducible of codimension
		\begin{center}
			$m+p-q+1 = m+u+1$
		\end{center}
	in $(\mathbb{A}^3)_m$.
	Moreover,
	we have $U_m \cap \mathbf{V}(y_q,z_q) = V_m$.
	By the assertion (b),
	we have
		\begin{center}
			$\mathrm{codim}_{(\mathbb{A}^3)_m}\ V_m > m+u+1$.
		\end{center}
	We note that $U_m$ is defined by the ideal
	$L_{pqq} + \langle g_{pqq}^{(2p)},...,g_{pqq}^{(m)} \rangle$,
	which is generated by $p+q+q+(m-2p+1) = m+u+1$ elements.
	Thus,
	by Krull's height theorem,
	$V_m$ contains no irreducible components of $U_m$.
	Hence
		\begin{center}
			$U_m = \overline{U_m\cap \mathbf{D}(y_q)}$
		\end{center}
	is irreducible of
		\begin{center}
			$\mathrm{codim}_{(\mathbb{A}^3)_m}\ U_m = m+u+1$.
		\end{center}
	This complete the proof of (c) and the proof of the proposition.
	\end{proof}
	\begin{proof}[\indent\sc Proof of Lemma \ref{irreducibility of Z_m^0}.]
	First,
	we consider the case $m = 5$.
	Then we have
		\begin{align*}
			Z_m^0 = \mathbf{V}(L_{322})
		\end{align*}
	by Remark \ref{D_4^1におけるI_m^0の変形},
	and this closed subset is irreducible of codimension $3+2+2 = 7 = m+2$.
	
	Next,
	we prove the statement for $m \geq 6$.
	Let $u$ be the positive integer with $6u \leq m < 6(u+1)$.
	If $u = 1$,
	then by Remark \ref{D_4^1におけるI_m^0の変形} and Proposition \ref{center irreducible components of D_4^1},
	$Z_m^0 = \mathbf{V}(L_{322}+\langle g^{(0)},...,g^{(m)} \rangle) = U_m$ is irreducible of codimension $m+u+1 = m+2$.
	Hence this case was proven.
	In the following,
	we assume $u \geq 2$.
	We note that $\mathbf{V}(L_{322})$ is equal to
		\begin{align}\label{decomposition of singular locus}
			\bigcup_{u'=2}^{2u-1} \Big(\mathbf{V}(L_{3,u',u'}) \cap \big(\mathbf{D}(y_{u'})\cup \mathbf{D}(z_{u'})\big)\Big) \cup \mathbf{V}(L_{3,2u,2u}).
		\end{align}
	Since $Z_m^0$ can be written as $Z_m^0 \cap \mathbf{V}(L_{322})$ by Remark \ref{D_4^1におけるI_m^0の変形},
	it is the union of the following sets:
		\begin{itemize}
			\item[($\alpha$)]
				$Y_{u'} := Z_m^0 \cap \Big(\mathbf{V}(L_{3,u',u'}) \cap \big(\mathbf{D}(y_{u'})\cup \mathbf{D}(z_{u'})\big)\Big)$ ($u' = 2,3,...,2u-1$),
			\item[($\beta$)]
				$Y_{2u} := Z_m^0 \cap \mathbf{V}(L_{3,2u,2u})$.
		\end{itemize}
	Now,
	we will see the following claim.
	\begin{claim}
		\begin{itemize}
			\item[(a)]
				$Y_2$ is irreducible of codimension $m+2$ in $(\mathbb{A}^3)_m$.
			\item[(b)]
				For $u' = 3,4,...,2u$,
				$Y_{u'}$ is of codimension greater than $m+2$ in $(\mathbb{A}^3)_m$.
		\end{itemize}
	\end{claim}
	Note that $L_{322} + \langle g^{(0)},...,g^{(m)} \rangle = L_{322} + \langle g_{322}^{(6)},...,g_{322}^{(m)} \rangle$ by Remark \ref{座標関数を含むイデアルの簡略化} and by Remark \ref{D_4^1におけるI_m^0の変形}.
	This ideal is generated by $3+2+2+(m-5) = m+2$ elements,
	so we see that any irreducible component of $Z_m^0 = \mathbf{V}(I_m^0)$ is of codimension at most $m+2$ in $(\mathbb{A}^3)_m$.
	Thus if we can prove the above claim,
	then the proof is complete.
	
	Before proving the claim,
	we give another set of defining equations of $Z_m^0 \cap \mathbf{V}(L_{3,u',u'})$ for $u' = 2,3,...,2u$.
	Note that
		\begin{align*}
			Z_m^0 \cap \mathbf{V}(L_{3,u',u'})	&= \mathbf{V}(L_{3,u',u'}+\langle g^{(6)}_{3,u',u'},...,g^{(m)}_{3,u',u'} \rangle)
		\end{align*}
	by Remark \ref{D_4^1におけるI_m^0の変形}.
	Let us show that
		\begin{align}\label{L_{3,u'u'}の座標函数の詳細な計算}
			\sqrt{{L_{3,u',u'} + \langle g_{3,u',u'}^{(6)},...,g_{3,u',u'}^{(m)} \rangle}}	&= \sqrt{L_{\lceil 3u'/2 \rceil,u',u'} + \langle g_{3,u',u'}^{(6)},...,g_{3,u',u'}^{(m)} \rangle}\\
				&= \sqrt{L_{\lceil 3u'/2 \rceil,u',u'} + \langle g_{\lceil 3u'/2 \rceil,u',u'}^{(6)},...,g_{\lceil 3u'/2 \rceil,u',u'}^{(m)} \rangle}.
		\end{align}
	For the case $u' = 2$,
	then $\lceil 3u'/2 \rceil = 3$ and hence the assertion holds.
	
	For the cases $u' \geq 3$,
	it is suffices to show that
		\begin{align*}
			x_{l/2} \in \sqrt{{L_{3,u',u'} + \langle g_{3,u',u'}^{(6)},...,g_{3,u',u'}^{(m)} \rangle}}
		\end{align*}
	for $l = 6, 8, ..., 2(\lceil 3u'/2 \rceil-1)$.
	We prove this by induction.
	If $l = 6$,
	then $l = 6 < 9 \leq 3u'\ (\leq m)$ and hence $g^{(6)}_{3,u',u'} = x_3^2$ by Lemma \ref{reduction of f^{(l)} g^{(l)} modulo L_{pqr}}(b),
	and hence
		\begin{center}
			$x_3 \in \sqrt{{L_{3,u',u'} + \langle g_{3,u',u'}^{(6)},...,g_{3,u',u'}^{(m)} \rangle}}$.
		\end{center}
	
	For $l \geq 8$,
	we assume that
		\begin{center}
			$x_{l'/2} \in \sqrt{{L_{3,u',u'} + \langle g_{3,u',u'}^{(6)},...,g_{3,u',u'}^{(m)} \rangle}}$
		\end{center}
	holds for even integers $l' < l$.
	Then we have
		\begin{center}
			$\sqrt{{L_{3,u',u'} + \langle g_{3,u',u'}^{(6)},...,g_{3,u',u'}^{(m)} \rangle}} = \sqrt{L_{l/2,u',u'} + \langle g_{l/2,u',u'}^{(6)},...,g_{l/2,u',u'}^{(m)} \rangle}$.
		\end{center}
	Since $l/2 < \lceil 3u'/2 \rceil < 2u'$ for $u' > 0$,
	we have $l \leq 2(\lceil 3u'/2 \rceil-1) \leq 2((3u'+1)/2-1) = 3u'+1 -2 = 3u'-1 < 3u'\ (\leq m)$.
	Hence we have
		\begin{center}
			$g^{(l)}_{l/2,u',u'} = x_{l/2}^2$
		\end{center}
	and
		\begin{center}
			$x_{l/2} \in \sqrt{L_{l/2,u',u'} + \langle g_{l/2,u',u'}^{(6)},...,g_{l/2,u',u'}^{(m)} \rangle}$.
		\end{center}
	Therefore,
	we have
		\begin{align*}
			x_{l/2} \in \sqrt{{L_{3,u',u'} + \langle g_{3,u',u'}^{(6)},...,g_{3,u',u'}^{(m)} \rangle}}
		\end{align*}
	for $l = 6, 8, ..., 2(\lceil 3u'/2 \rceil-1)$ and (2) holds.
	
	Now,
	we prove the statements (a) and (b) of the claim with $u' \neq 2u$.
	If $u'$ is even,
	then,
	for $u'' = u'/2$,
	we have
		\begin{align*}
			\overline{Z_m^0 \cap \mathbf{V}(L_{3,u',u'}) \cap \mathbf{D}(y_{u'})} &= \overline{\mathbf{V}(L_{3u'',2u'',2u''} + \langle g_{3u'',2u'',2u''}^{(6)},...,g_{3u'',2u'',2u''}^{(m)} \rangle) \cap \mathbf{D}(y_{2u''})},\\
			\overline{Z_m^0 \cap \mathbf{V}(L_{3,u',u'}) \cap \mathbf{D}(z_{u'})} &= \overline{\mathbf{V}(L_{3u'',2u'',2u''} +\langle g_{3u'',2u'',2u''}^{(6)},...,g_{3u'',2u'',2u''}^{(m)} \rangle) \cap \mathbf{D}(z_{2u''})}
		\end{align*}
	by \eqref{L_{3,u'u'}の座標函数の詳細な計算} and these sets coincide and are irreducible of codimension $m - 3u'' +2(2u'') +1 = m+u''+1$ by Lemma \ref{open irreducible and codimension}(d) with $p = 3u''$ and $q = r = 2u''$.
	Hence,
	if $u' = 2$,
	i.e.,
	$u'' = 1$,
	then this set is of codimension $m+2$.
	If $u' \geq 4$ i.e.,
	$u'' \geq 2$,
	then this set is of codimension $m+u''+1 > m+2$.
	Thus,
	the case in which $u'$ is even is complete.
	
	Next,
	if $u'$ is odd,
	then we have
		\begin{align*}
			Z_m^0 \cap \mathbf{V}(L_{3,u',u'})
			=&\ \mathbf{V}(L_{\frac{3u'+1}{2},u',u'}+\langle g_{\frac{3u'+1}{2},u',u'}^{(6)},...,g_{\frac{3u'+1}{2},u',u'}^{(m)} \rangle).
		\end{align*}
	by \eqref{L_{3,u'u'}の座標函数の詳細な計算}.
	By Lemma \ref{reduction of f^{(l)} g^{(l)} modulo L_{pqr}}(a) and (d) with $p = (3u'+1)/2, q = u'$ and $r = u'$,
	we have $g_{\frac{3u'+1}{2},u',u'}^{l} = 0$ for $l < 3u'$ and $g_{\frac{3u'+1}{2},u',u'}^{(3u')} = y_{u'}z_{u'}(y_{u'}+z_{u'})$.
	Thus
		\begin{align*}\displaystyle
			Z_m^0 \cap \mathbf{V}(L_{3,u',u'})	&= 
			\bigcup_{i=1,2,3} \left(\mathbf{V}(L_{\frac{3u'+1}{2},u',u'} + \langle h_{u',i} \rangle + \langle g_{\frac{3u'+1}{2},u',u'}^{(3u'+1)},...,g_{\frac{3u'+1}{2},u',u'}^{(m)} \rangle) \right)
		\end{align*}
	where $h_{u',1} = y_{u'}$,
	$h_{u',2} = z_{u'}$ and $h_{u',3} = y_{u'}+z_{u'}$.
	Therefore,
		\begin{align*}
			 & \overline{Z_m^0 \cap \mathbf{V}(L_{3,u',u'}) \cap \mathbf{D}(y_{u'})}\\
			=& \bigcup_{i = 2,3} \overline{\mathbf{V}(L_{\frac{3u'+1}{2},u',u'}+\langle h_{u',i} \rangle +\langle g_{\frac{3u'+1}{2},u',u'}^{(3u'+1)},...,g_{\frac{3u'+1}{2},u',u'}^{(m)} \rangle) \cap \mathbf{D}(y_{u'})}
		\end{align*}
	and both summands are irreducible of codimension $m+p-q+1 > m+2$ by Lemma \ref{open irreducible and codimension}(b) with $p = (3u'+1)/2$,
	$q = u'$ and $r = u'+1$ and Lemma \ref{open irreducible and codimension}(c) with $p = (3u'+1)/2$ and $q = r = u'$.
	By the symmetry,
	we have the same conclusion for
		\begin{align*}
			 \overline{Z_m^0 \cap \mathbf{V}(L_{3,u',u'}) \cap \mathbf{D}(z_{u'})}.
		\end{align*}
	Hence,
	for $2 < u' \leq 2u-1$,
	$Y_{u'}$
	is of codimension greater than $m+2$.
	Thus,
	the case in which $u'$ is odd is complete.

	Finally,
	we prove (b) with $u' = 2u$.
	By \eqref{L_{3,u'u'}の座標函数の詳細な計算},
		\begin{align*}
			Y_{2u} = Z_m^0 \cap \mathbf{V}(L_{3,2u,2u})	&=	\mathbf{V}(L_{3u,2u,2u} + \langle g_{3u,2u,2u}^{(6)},...,g_{3u,2u,2u}^{(m)} \rangle).
		\end{align*}
	Then this closed subset is irreducible of codimension $m+u+1$ by Proposition \ref{center irreducible components of D_4^1}.
	Furthermore,
	we have $u > 1$ by assumption.
	Hence $Y_{2u}$ is of codimension $m+u+1 > m+2$.
	\end{proof}
%
%
%
%
We remark on the symmetries of the irreducible components of the singular fiber.
	\begin{Lemma}\label{symmetries of $X$}
	Assume $m \geq 3$.
	The closed subsets $Z_m^0, Z_m^1, Z_m^2$ and $Z_m^3$ are mapped to another by $\psi_1$ and $\psi_2$ (see Notation \ref{def of symmetric morphism of D_4^1}) as follows:
		\begin{itemize}
		\item[(a)]  $\psi_1(Z_m^0) = Z_m^0$,
					$\psi_1(Z_m^1) = Z_m^2$,
					$\psi_1(Z_m^2) = Z_m^1$ and $\psi_1(Z_m^3) = Z_m^3$.
		\item[(b)]  $\psi_2(Z_m^0) = Z_m^0$,
					$\psi_2(Z_m^1) = Z_m^1$,
					$\psi_2(Z_m^2) = Z_m^3$ and $\psi_2(Z_m^3) = Z_m^2$.
		\end{itemize}
	\end{Lemma}
By Lemma \ref{symmetric morphisms of D_4^1},
we can prove this lemma in the same way as Lemma \ref{symmetries of $S$}.
	\begin{Prop}\label{rough decomposition of D_4^1}
		For $m \geq 3$,
		we have
			\begin{align*}
				S_m^0	&= Z_m^0 \cup Z_m^1 \cup Z_m^2 \cup Z_m^3.
			\end{align*}
		Moreover,
		$Z_m^1$, $Z_m^2$ and $Z_m^3$ are pairwise distinct.
	\end{Prop}
The proof of this proposition is the same as that of Proposition \ref{rough decomposition of D_4^0}

Now,
we give the irreducible decomposition of the singular fiber $S_m^0$.
The following proposition gives the irreducible decomposition for small $m$.
	\begin{Prop}\label{小さな次数でのD_4^1の既約分解}
		For $0 \leq m \leq 4$,
		the irreducible decomposition of the singular fiber $S_m^0$ is as follows.
			\begin{enumerate}
				\item[(a)]
					$S_0^0 = \mathbf{V}(L_{111})$,
				\item[(b)]
					$S_1^0 = \mathbf{V}(L_{111})$,
				\item[(c)]
					$S_2^0 = \mathbf{V}(L_{211})$,
				\item[(d)]
					$S_3^0 = Z_3^1 \cup Z_3^2 \cup Z_3^3$,
				\item[(e)]
					$S_4^0 = Z_4^1 \cup Z_4^2 \cup Z_4^3$,
			\end{enumerate}
		where $Z_m^i$'s are as in Definition \ref{definition of the irreducible components of singular fiber of D_4^1}.
		Moreover,
		for $1 \leq m \leq 4$,
		the codimension of any irreducible component $Z \subseteq S_m^0$ is
			\begin{center}
				$\mathrm{codim}_{(\mathbb{A}^3)_m}\ Z = m+2$.
			\end{center}
	\end{Prop}
This proposition can be proved in the same way as Proposition \ref{小さな次数でのD_4^0の既約分解}.
	\begin{Thm}\label{irreducible decompositions in D_4^1}
	For $m \geq 5$,
	the irreducible decomposition of the singular fiber $S_m^0$ is
		\begin{center}
		$S_m^0 = Z_m^0 \cup Z_m^1 \cup Z_m^2 \cup Z_m^3$,
		\end{center}
	where $Z_m^i$'s are as in Definition \ref{definition of the irreducible components of singular fiber of D_4^1}.
	\end{Thm}
	\begin{proof}
	By Proposition \ref{rough decomposition of D_4^1},
		\begin{alignat*}{7}
		S_m^0 	=& Z_m^0 \cup Z_m^1 \cup Z_m^2 \cup Z_m^3.
		\end{alignat*}
	Moreover,
	by Corollary \ref{calculation of the codimension of three irreducible components of S_m^0 of D_4^1} and Lemma \ref{irreducibility of Z_m^0},
	$Z_m^0$,
	$Z_m^1$,
	$Z_m^2$ and $Z_m^3$ are irreducible of codimensions $m+2$ in $(\mathbb{A}^3)_m$.
	Hence we only have to show that $Z_m^i$,
	for $0 \leq i \leq 3$,
	are pairwise distinct.
	
	For $m \geq 5$,
	we have $Z_m^0 \cap (\mathbf{D}(y_1) \cup \mathbf{D}(z_1)) = \emptyset$ while $Z_m^i \cap (\mathbf{D}(y_1) \cup \mathbf{D}(z_1)) \neq \emptyset$ for $1 \leq  i \leq 3$,
	so $Z_m^0$ is different from $Z_m^1$,
	$Z_m^2$ and $Z_m^3$.
	Moreover,
	by Proposition \ref{rough decomposition of D_4^1} $Z_m^1$,
	$Z_m^2$ and $Z_m^3$ are pairwise distinct.
	This complete the proof.
	\end{proof}

To conclude this paper,
we prove the following theorem on the jet schemes of $D_4^1$-singularities,
which is analogous to Theorem \ref{main theorem of D_4^0} and the characteristic $0$ case \cite[Theorem 3.17]{Ky}.
	\begin{Thm}\label{Config of D_4^1}
	Let $k$ be an algebraically closed field of characteristic $2$,
	$S \subset \mathbb{A}^3$ the surface defined by $g = x^2 + y^2z + yz^2 + xyz$ in the affine space over $k$,
	$S_m^0$ the singular fiber of the $m$-th jet scheme $S_m$ with $m \geq 5$ and $Z_m^0,...,Z_m^3$ its irreducible components as in Definition \ref{definition of the irreducible components of singular fiber of D_4^1}.
		\begin{itemize}
			\item[(a)]
				For $0 \leq i < j \leq 3$,
				$Z_m^i \cap Z_m^j \subsetneq Z_m^0$.
			\item[(b)]
				For $1 \leq i, j \leq 3$ with $i \neq j$,
				$Z_m^0 \cap Z_m^i \not\subseteq Z_m^0 \cap Z_m^j$.
			\item[(c)]
				For $1 \leq i, j \leq 3$ with $i \neq j$,
				$Z_m^i \cap Z_m^j \subsetneq Z_m^0 \cap Z_m^i$.
			\item[(d)]
				For $1 \leq i < j \leq 3$ and $1 \leq l \leq 3$,
				$Z_m^0 \cap Z_m^l \not\subseteq Z_m^i \cap Z_m^j$.
		\end{itemize}
	In particular,
	for $0 \leq i<j \leq 3$,
	$Z_m^i \cap Z_m^j$ is maximal in  $\{Z_m^i \cap Z_m^j \mid i, j \in \{0,1,2,3\}, i \neq j\}$ with respect to the inclusion relation if and only if $(i,j) = (0,1), (0,2), (0,3)$.
	\end{Thm}
	\begin{proof}
	We argue as in the proof of Theorem \ref{main theorem of D_4^0}.
	
	(a)
	Note that
		\begin{center}
		$g^{(4)} \equiv f^{(4)}$
		\end{center}
	modulo $L_{221}$ or $L_{212}$ for $f = x^2+y^2z+yz^2$ since terms from $\mathbf{x}_2\mathbf{y}_2\mathbf{z}_1$ or $\mathbf{x}_2\mathbf{y}_1\mathbf{z}_2$ are of degree at least $5$,
	where $\mathbf{x}_i, \mathbf{y}_i, \mathbf{z}_i$ are as in Notation \ref{modulo polynomials}.
	Hence we can prove $Z_m^i \cap Z_m^j \subsetneq Z_m^0$ as in the proof of Theorem \ref{main theorem of D_4^0}.
	\medskip\\
	(b)
	By Lemma \ref{symmetries of $X$},
	it suffices to show that
		\begin{itemize}
			\item[$\rm(\hspace{.18em}i\hspace{.18em})$]
				$(0,t^2,t^2) \in Z_m^0 \cap Z_m^3$.
			\item[$\rm(\hspace{.08em}ii\hspace{.08em})$]
				$(0,t^2,t^2) \notin Z_m^0 \cap Z_m^1$.
		\end{itemize}
	Let $P = (0,t^2,t^2)$.
	We have $g(P) = 0^2 + (t^2)^2\cdot t^2 + t^2\cdot (t^2)^2 + 0\cdot t^2 \cdot t^2 = 2t^6 = 0$,
	hence we have $P \in S_m$.
	
	$\rm(\hspace{.18em}i\hspace{.18em})$
	Let us prove $P \in Z_m^0 \cap Z_m^3$.
	We note that $P$ corresponds to $x_{\alpha} = y_{\beta} = z_{\beta} = 0$ for $\alpha \in \{ 0,...,m\}$ and $\beta \in \{0,1,3,...,m\}$ and $y_2 = z_2 = 1$.
	Thus we have $P \in Z_m^0 = S_m \cap \mathbf{V}(L_{222})$.
	We put $P_s = (0,st+t^2,st+t^2)$ for $s \neq 0$,
	and then we have $g(P_s) = 0$ and $P_s$ corresponds to $x_{\alpha} = y_{\beta} = z_{\beta} = 0$ for $\alpha \in \{ 0,...,m\}$ and $\beta \in \{0,3,...,m\}$,
	$y_1 = z_1 = s \neq 0$ and $y_2 = z_2 = 1$.
	Hence we have $P_s \in \mathbf{V}(L_{211} + \langle y_1+z_1 \rangle) \cap \mathbf{D}(y_1)$.
	Taking the limit $s \rightarrow 0$,
	we have $P_s \rightarrow P$ and $P \in \overline{\mathbf{V}(L_{211} + \langle y_1+z_1 \rangle) \cap \mathbf{D}(y_1)} = Z_m^3$. 
	Therefore,
	we have $P \in Z_m^0 \cap Z_m^3$.
	
	$\rm(\hspace{.08em}ii\hspace{.08em})$
	We prove $P \notin Z_m^0 \cap Z_m^1$.
	We have
		\begin{center}
		$g^{(5)}_{221} = y_2^2z_1 + y_3z_1^2 + x_2y_2z_1 = z_1(y_2^2 + y_3z_1 + x_2y_2)$,
		\end{center}
	hence we have $y_2^2+y_3z_1+x_2y_2 \in J_m^1\cdot (R_m)_{z_1} \cap R_m = I_m^1$.
	Since $x_2, z_1 \in L_{322} \subseteq \sqrt{I_m^0}$ (see Remark \ref{D_4^1におけるI_m^0の変形}),
	we have $y_2 \in \sqrt{I_m^0 + I_m^1}$.
	Therefore we have $P \notin Z_m^0 \cap Z_m^1$.
	\medskip\\
	(c), (d)
	The assertions are proven in the same way as in the proof of Theorem \ref{main theorem of D_4^0}(c) and (d).
	\end{proof}
	\begin{Coro}
	For $m \geq 5$,
	the graph $\Gamma(S_m^0)$ obtained by Construction \ref{graph constructed by the singular fiber} is the resolution graph of a $D_4$-type singularity.
	\end{Coro}

\end{document}